\documentclass[11pt]{amsart}
\usepackage{amsmath,amssymb,euscript,mathrsfs,pstricks}
\usepackage[small,heads=LaTeX,nohug]{diagrams}
\usepackage[colorlinks=true,linkcolor=black,citecolor=black,urlcolor=black]{hyperref}

\psset{unit=1pt}
\psset{arrowsize=4pt 1}
\psset{linewidth=.5pt}

\newcommand{\define}{\textbf}

\renewcommand{\setminus}{\smallsetminus}
\renewcommand{\phi}{\varphi}

\newcommand{\exterior}{\textstyle\bigwedge}
\newcommand{\isom}{\cong}
\renewcommand{\tilde}{\widetilde}
\renewcommand{\hat}{\widehat}
\renewcommand{\bar}{\overline}

\newcommand{\<}{\langle}
\renewcommand{\>}{\rangle}

\renewcommand{\O}{\mathcal{O}}
\newcommand{\Fl}{{Fl}}
\newcommand{\GG}{\mathcal{G}}
\newcommand{\QQ}{\mathcal{Q}}
\newcommand{\FFl}{\mathbf{Fl}}

\newcommand{\OOmega}{\mathbf{\Omega}}

\newcommand{\C}{\mathbb{C}}
\newcommand{\Q}{\mathbb{Q}}

\newcommand{\Z}{\mathbb{Z}}

\renewcommand{\P}{\mathbb{P}}
\newcommand{\A}{\mathbb{A}}

\newcommand{\lieg}{\mathfrak{g}}
\newcommand{\liep}{\mathfrak{p}}
\newcommand{\lieb}{\mathfrak{b}}
\newcommand{\liet}{\mathfrak{t}}

\newcommand{\liesl}{\mathfrak{sl}}

\newcommand{\GP}{\mathfrak{G}}

\DeclareMathOperator{\Char}{char}
\DeclareMathOperator{\Sym}{Sym}

\DeclareMathOperator{\rk}{rk}
\DeclareMathOperator{\diag}{diag}

\DeclareMathOperator{\Aut}{Aut}

\newtheorem{theorem}{Theorem}
\newtheorem{lemma}[theorem]{Lemma}
\newtheorem{proposition}[theorem]{Proposition}
\newtheorem{corollary}[theorem]{Corollary}

\theoremstyle{definition}
\newtheorem{definition}[theorem]{Definition}
\newtheorem{remark}[theorem]{Remark}
\newtheorem{example}[theorem]{Example}

\newenvironment{problem}[1]
 {\begin{list}{}{\leftmargin25pt\rightmargin25pt}\item{}\noindent\textit{#1}}
 {\end{list}\smallskip}

\numberwithin{theorem}{section}

\numberwithin{equation}{section}

\begin{document}

\title{Chern class formulas for $G_2$ Schubert loci}
\author{Dave Anderson}
\address{Department of Mathematics\\University of Michigan\\Ann Arbor, MI 48109}
\email{dandersn@umich.edu}
\date{February 1, 2010}
\thanks{This work was partially supported by NSF Grants DMS-0502170 and DMS-0902967.}

\begin{abstract}
We define degeneracy loci for vector bundles with structure group $G_2$, and give formulas for their cohomology (or Chow) classes in terms of the Chern classes of the bundles involved.  When the base is a point, such formulas are part of the theory for rational homogeneous spaces developed by Bernstein--Gelfand--Gelfand and Demazure.  This has been extended to the setting of general algebraic geometry by Giambelli--Thom--Porteous, Kempf--Laksov, and Fulton in classical types; the present work carries out the analogous program in type $G_2$.  We include explicit descriptions of the $G_2$ flag variety and its Schubert varieties, and several computations, including one that answers a question of W. Graham.

In appendices, we collect some facts from representation theory and compute the Chow rings of quadric bundles, correcting an error in \cite{eg}.
\end{abstract}
\maketitle

\setcounter{tocdepth}{1}
\tableofcontents

\section{Introduction} \label{ch:intro}

Let $V$ be an $n$-dimensional vector space.  The \emph{flag variety} $\Fl(V)$ parametrizes all complete flags in $V$, i.e., saturated chains of subspaces $E_\bullet = (E_1 \subset E_2 \subset \cdots \subset E_n = V)$ (with $\dim E_i = i$).  Fixing a flag $F_\bullet$ allows one to define \emph{Schubert varieties} in $\Fl(V)$ as the loci of flags satisfying certain incidence conditions with $F_\bullet$; there is one such Schubert variety for each permutation of $\{1,\ldots,n\}$.  This generalizes naturally to the case where $V$ is a vector bundle and $F_\bullet$ is a flag of subbundles.  Here one has a \emph{flag bundle} $\FFl(V)$ over the base variety, whose fibers are flag varieties, with \emph{Schubert loci} defined similarly by incidence conditions.  Formulas for the cohomology classes of these Schubert loci, as polynomials in the Chern classes of the bundles involved, include the classical Thom--Porteous--Giambelli and Kempf--Laksov formulas (see \cite{flags}).

The above situation is ``type $A$,'' in the sense that $\Fl(V)$ is isomorphic to the homogeneous space $SL_n/B$ (with $B$ the subgroup of upper-triangular matrices).  There are straightforward generalizations to the other classical types ($B$, $C$, $D$): here the vector bundle $V$ is equipped with a symplectic or nondegenerate symmetric bilinear form, and the flags are required to be isotropic with respect to the given form.  Schubert loci are defined as before, with one for each element of the corresponding Weyl group.  The problem of finding formulas for their cohomology classes has been studied by Harris--Tu \cite{ht}, J\'ozefiak--Lascoux--Pragacz \cite{jlp}, and Fulton \cite{clgp,orthosymp}, among others.

One is naturally led to consider the analogous problem in the five remaining Lie types.  In exceptional types, however, it is not so obvious how the Lie-theoretic geometry of $G/B$ generalizes to the setting of vector bundles in algebraic geometry.  The primary goal of this article is to carry this out for type $G_2$.

To give a better idea of the difference between classical and exceptional types, let us describe the classical problem in slightly more detail.  The flag bundles are the universal cases of general \emph{degeneracy locus} problems in algebraic geometry.  Specifically, let $V$ be a vector bundle of rank $n$ on a variety $X$, and let $\phi:V\otimes V \to k$ be a symplectic or nondegenerate symmetric bilinear form (or the zero form).  If $E_\bullet$ and $F_\bullet$ are general flags of isotropic subbundles of $V$, the problem is to find formulas in $H^*X$ for the \emph{degeneracy locus}
\begin{eqnarray*}
D_w = \{ x\in X \,|\, \dim(F_p(x) \cap E_q(x))\geq r_{w\,w_0}(q,p) \},
\end{eqnarray*}
in terms of the Chern classes of the line bundles $E_{q}/E_{q-1}$ and $F_p/F_{p-1}$, for all $p$ and $q$.  (Here $w$ is an element of the Weyl group, considered as a permutation via an embedding in the symmetric group $S_n$; $w_0$ is the longest element, corresponding to the permutation $n\; n-1\; \cdots\;1$; and $r_w(q,p) = \#\{i\leq q\,|\, w(i)\leq p\}$ is a nonnegative integer depending on $w$, $p$, and $q$.)  Such formulas have a wide range of applications: for example, they appear in the theory of special divisors and variation of Hodge structure on curves in algebraic geometry \cite{ht,pp}, and they are used to study singularities of smooth maps in differential geometry (work of Feh\'er and Rim\'anyi, e.g., \cite{fr-ss}).  They are also of interest in combinatorics (e.g., work of Lascoux--Sch\"utzenberger, Fomin--Kirillov, Pragacz, Kresch--Tamvakis).  See \cite{fp} for a more detailed account of the history.

In this article, we pose and solve the corresponding problem in type $G_2$:

\begin{problem}
{Let $V \to X$ be a vector bundle of rank $7$, equipped with a nondegenerate alternating trilinear form $\gamma:\exterior^3 V \to L$, for a line bundle $L$.  Let $E_\bullet$ and $F_\bullet$ be general flags of $\gamma$-isotropic subbundles of $V$, and let
\begin{eqnarray*}
D_w = \{ x\in X \,|\, \dim(F_p(x) \cap E_q(x))\geq r_{w\,w_0}(q,p) \},
\end{eqnarray*}
where $w$ is an element of the Weyl group for $G_2$ (the dihedral group with $12$ elements).  Find a formula for $[D_w]$ in $H^*X$, in terms of the Chern classes of the bundles involved.}
\end{problem}

\noindent
The meaning of ``nondegenerate'' and ``$\gamma$-isotropic'' will be explained below (\S\S\ref{introsec:compatible}--\ref{introsec:isotropic}), as will the precise definition of $D_w$ (\S\ref{introsec:degloci}).  In order to establish the relation between group theory and geometry, we give descriptions of the $G_2$ flag variety and its Schubert subvarieties which appear to be new, although they will not surprise the experts (\S\ref{sec:topology}, \S\ref{sec:homogeneous}).  This is done in such a way as to make the transition to flag bundles natural.  We then give presentations of the cohomology rings of these flag bundles, including ones with integer coefficients (Theorem \ref{thm:integral-presentation}).  Finally, we prove formulas for the classes of Schubert varieties in flag bundles (\S\ref{sec:divdiff}); the formulas themselves are given in \cite[Appendix D.2]{thesis}.  We also discuss alternative formulas, answer a question of William Graham about the integrality of a certain rational cohomology class, and prove a result giving restrictions on candidates for ``$G_2$ Schubert polynomials'' (\S\ref{sec:variations}).

We also need a result on the integral cohomology of quadric bundles, which were studied in \cite{eg}.  Appendix \ref{ch:chow} corrects a small error in that article.

Various constructions of exceptional-type flag varieties have been given using techniques from algebra and representation theory; those appearing in \cite{lm2}, \cite{im}, and \cite{gari-e7} have a similar flavor to the one presented here.  A key feature of our description is that the data parametrized by the $G_2$ flag variety naturally determine a \emph{complete} flag in a $7$-dimensional vector space, much as isotropic flags in classical types determine complete flags by taking orthogonal complements.  The fundamental facts that make this work are Proposition \ref{prop:3d} and its cousins, Corollary \ref{cor:231} and Propositions \ref{prop:3d-bundle} and \ref{prop:231-bundle}.

Formulas for degeneracy loci are closely related to Giambelli formulas for equivariant classes of Schubert varieties in the equivariant cohomology of the corresponding flag variety.  We will usually use the language of degeneracy loci, but we discuss the connection with equivariant cohomology in \S\ref{introsec:equivariant}.  In brief, the two perspectives are equivalent when $\det V$ and $L$ are trivial line bundles.

Another notion of degeneracy loci is often useful, where one is given a map of vector bundles $\phi:E \to F$ on $X$, possibly possessing some kind of symmetry, and one is interested in the locus where $\phi$ drops rank.  This is the situation considered in \cite{ht}, for example, with $F=E^*$ and symmetric or skew-symmetric maps.  We investigate the $G_2$ analogue of this problem in \cite{morphisms}.


When the base $X$ is a point, so $V$ is a vector space and the flag bundle is just the flag variety $G/B$, most of the results have been known for some time; essentially everything can be done using the general tools of Lie theory.  For example, a presentation of $H^*(G/B,\Z)$ was given by Bott and Samelson \cite{bs}, and (different) formulas for Schubert classes in $H^*(G/B,\Q)$ appear in \cite{bgg}.  Since this article also aims to present a concrete, unified perspective on the $G_2$ flag variety, accessible to general algebraic geometers, we wish to emphasize geometry over Lie theory: we are describing a geometric situation from which type-$G_2$ groups arise naturally.  Reflecting this perspective, we postpone the Lie- and representation-theoretic arguments to Appendix \ref{ch:liethy}.  We shall use some of the notation and results of this appendix throughout the article, though, so the reader less familiar with Lie theory is advised to skim at least \S\ref{sec:lie-general}, \S\ref{sec:weyl}, and \S\ref{sec:borel}.

\medskip

\noindent
{\it Notation and conventions.}  Unless otherwise indicated, the base field $k$ will have characteristic not $2$ and be algebraically closed (although a quadratic extension of the prime field usually suffices).  When $\Char(k)=2$, several of our definitions and results about forms and octonions break down.  However, most of the other main results hold in arbitrary characteristic, including the description of the $G_2$ flag variety and its cohomology, the degeneracy locus formulas, and the parametrizations of Schubert cells; see \cite[Chapter 6]{thesis} for details in characteristic $2$.

Angle brackets denote the span of enclosed vectors: $\<x,y,z\> := \mathrm{span}\{x,y,z\}$.  

For a vector bundle $V$ on $X$ and a point $x\in X$, $V(x)$ denotes the fiber over $x$.  If $X\to Y$ is a morphism and $V$ is a vector bundle on $Y$, we will often write $V$ for the vector bundle pulled back to $X$.  If $V$ is a vector space and $E$ is a subspace, $[E]$ denotes the corresponding point in an appropriate Grassmannian.

We generally use the notation and language of (singular) cohomology, but this should be read as Chow cohomology for ground fields other than $\C$.  (Since the varieties whose cohomology we compute are rational homogeneous spaces or fibered in homogeneous spaces, the distinction is not significant.)

\medskip
\noindent
{\it Acknowledgements.}  This work is part of my Ph.~D.~thesis, and it is a pleasure to thank William Fulton for his encouragement in this project and careful readings of earlier drafts.  Conversations and correspondence with many people have benefitted me; in particular, I would like to thank Robert Bryant, Skip Garibaldi, William Graham, Sam Payne, and Ravi Vakil.  Thanks also to an anonymous referee for comments on the manuscript.

\section{Overview} \label{sec:overview}

We begin with an overview of our description of the $G_2$ flag variety and statements of the main results.  Proofs and details are given in later sections.

\subsection{Compatible forms} \label{introsec:compatible}

Let $V$ be a $k$-vector space.  Let $\beta$ be a nondegenerate symmetric bilinear form on $V$, and let $\gamma$ be an alternating trilinear form, i.e., $\gamma:\exterior^3 V \to k$.  Write $v\mapsto v^\dag$ for the isomorphism $V \to V^*$ defined by $\beta$, and $\phi\mapsto\phi^\dag$ for the inverse map $V^*\to V$.  (Explcitly, these are defined by $v^\dag(u) = \beta(v,u)$ and $\phi(u) = (\phi^\dag,u)$ for any $u\in V$.)  Our constructions are based on the following definitions:
\begin{definition} \label{def:compatible}
Call the forms $\gamma$ and $\beta$ \define{compatible} if
\begin{eqnarray} \label{eqn:compatible}
2\,\gamma(u, v, \gamma(u, v, \cdot)^\dag ) = \beta(u,u)\beta(v,v) - \beta(u,v)^2
\end{eqnarray}
for all $u,v\in V$.  An alternating trilinear form $\gamma:\exterior^3 V \to k$ is \define{nondegenerate} if there exists a compatible nondegenerate symmetric bilinear form on $V$.
\end{definition}

The meaning of the strange-looking relation \eqref{eqn:compatible} will be explained in \S\ref{sec:oct}; see Proposition \ref{prop:compatible-composition}.  (The factor of $2$ is due to our convention that a quadratic norm and corresponding bilinear form are related by $\beta(u,u) = 2\,N(u)$.)  A pair of compatible forms is equivalent to a composition algebra structure on $k\oplus V$ (see \S \ref{sec:oct}).  Since a composition algebra must have dimension $1$, $2$, $4$, or $8$ over $k$ (by Hurwitz's theorem), it follows that nondegenerate trilinear forms exist only when $V$ has dimension $1$, $3$, or $7$.  In each case, there is an open dense $GL(V)$-orbit in $\exterior^3 V^*$ consisting of nondegenerate forms.  When $\dim V = 1$, the only alternating trilinear form is zero, and any nonzero bilinear form is compatible with it.  When $\dim V = 3$, an alternating trilinear form is a scalar multiple of the determinant, and given a nondegenerate bilinear form, it is easy to show that there is a unique compatible trilinear form up to sign.

When $\dim V = 7$, it is less obvious that $\exterior^3 V^*$ has an open $GL(V)$-orbit, especially if $\Char(k)=3$, but it is still true (Proposition \ref{prop:stabilizer}).  The choice of $\gamma$ determines $\beta$ uniquely up to scalar --- in fact, up to a cube root of unity (see Proposition \ref{prop:gamma-to-beta}).

Associated to any alternating trilinear form $\gamma$ on a seven-dimensional vector space $V$, there is a canonical map $B_\gamma:\Sym^2 V \to \exterior^7 V^*$, determining (up to scalar) a bilinear form $\beta_\gamma$.  We will give the formula for $\Char(k)\neq 3$ here.  Following Bryant \cite{bryant}, we define $B_\gamma$ by
\begin{eqnarray} \label{eqn:bryant-form}
B_\gamma(u,v) = -\frac{1}{3}\gamma(u,\cdot,\cdot)\wedge\gamma(v,\cdot,\cdot)\wedge\gamma,
\end{eqnarray}
where $\gamma(u,\cdot,\cdot):\exterior^2 V \to k$ is obtained by contracting $\gamma$ with $u$.  Choosing an isomorphism $\exterior^7 V^* \isom k$ yields a symmetric bilinear form $\beta_\gamma$.  If $\beta_\gamma$ is nondegenerate, then a scalar multiple of it is compatible with the trilinear form $\gamma$; thus $\gamma$ is nondegenerate if and only if $\beta_\gamma$ is nondegenerate.  The form $\beta_\gamma$ is defined in characteristic $3$, as well, and the statement still holds (see Lemma \ref{lemma:nondeg} and its proof).

\subsection{Isotropic spaces} \label{introsec:isotropic}

For the rest of this section, assume $\dim V = 7$.  Given a nondegenerate trilinear form $\gamma$ on $V$, say a subspace $F$ of dimension at least $2$ is \define{$\gamma$-isotropic} if $\gamma(u,v,\cdot) \equiv 0$ for all $u,v\in F$.  (That is, the map $F\otimes F \to V^*$ induced by $\gamma$ is zero.)  Say a vector or a $1$-dimensional subspace is $\gamma$-isotropic if it is contained in a $2$-dimensional $\gamma$-isotropic space.  If $\beta$ is a compatible bilinear form, every $\gamma$-isotropic subspace is also $\beta$-isotropic (Lemma \ref{lemma:gamma-to-beta}); as usual, this means $\beta$ restricts to zero on the subspace.  Since $\beta$ is nondegenerate, a maximal $\beta$-isotropic subspace has dimension $3$.

\begin{proposition} \label{prop:3d}
For any (nonzero) isotropic vector $u \in V$, the space 
\begin{eqnarray*}
E_u = \{ v\,|\, \<u,v\> \text{ is }\gamma\text{-isotropic} \}
\end{eqnarray*}
is three-dimensional and $\beta$-isotropic.  Moreover, every two-dimensional $\gamma$-isotropic subspace of $E_u$ contains $u$.
\end{proposition}

\noindent
The proof is given at the end of \S\ref{sec:oct-forms}.  The proposition implies that a maximal $\gamma$-isotropic subspace has dimension $2$, and motivates the central definition:
\begin{definition}
A \define{$\gamma$-isotropic flag} (or \define{$G_2$ flag}) in $V$ is a chain
\begin{eqnarray*}
F_1 \subset F_2 \subset V
\end{eqnarray*}
of $\gamma$-isotropic subspaces, of dimensions $1$ and $2$.  The variety parametrizing $\gamma$-isotropic flags is called the \define{$\gamma$-isotropic flag variety} (or \define{$G_2$ flag variety}), and denoted $\Fl_{\gamma}(V)$.
\end{definition}

The $\gamma$-isotropic flag variety is a smooth, six-dimensional projective variety (Proposition \ref{prop:gamma-flags}).  See \S\ref{sec:homogeneous} for its description as a homogeneous space.

Proposition \ref{prop:3d} shows that a $\gamma$-isotropic flag has a unique extension to a complete flag in $V$: set $F_3 = E_u$ for $u$ spanning $F_1$, and let $F_{7-i}$ be the orthogonal space $F_i^{\perp}$, with respect to a compatible form $\beta$.  (Since a compatible form is unique up to scalar, this is independent of the choice of $\beta$.)  This defines a closed immersion $\Fl_{\gamma}(V) \hookrightarrow \Fl_{\beta}(V) \subset \Fl(V)$, where $\Fl_{\beta}(V)$ and $\Fl(V)$ are the (classical) type $B$ and type $A$ flag varieties, respectively.

From the definition, there is a tautological sequence of vector bundles on $\Fl_{\gamma}(V)$,
\begin{eqnarray*}
S_1 \subset S_2 \subset V,
\end{eqnarray*}
and this extends to a \define{complete $\gamma$-isotropic flag} of bundles
\begin{eqnarray*}
S_1 \subset S_2 \subset S_3 \subset S_4 \subset S_5 \subset S_6 \subset V
\end{eqnarray*}
by the proposition.  Similarly, there are universal quotient bundles $Q_i = V/S_{7-i}$.

\subsection{Bundles} \label{introsec:bundles}

Now let $V \to X$ be a vector bundle of rank $7$, and let $L$ be a line bundle on $X$.  An alternating trilinear form $\gamma:\exterior^3 V \to L$ is \define{nondegenerate} if it is locally nondegenerate on fibers.  Equivalently, we may define the Bryant form $B_\gamma:\Sym^2 V \to \det V^* \otimes L^{\otimes 3}$ by Equation \eqref{eqn:bryant-form}, and $\gamma$ is nondegenerate if and only if $B_\gamma$ is (so $B_\gamma$ defines an isomorphism $V \isom V^* \otimes\det V^* \otimes L^{\otimes 3}$).

A subbundle $F$ of $V$ is \define{$\gamma$-isotropic} if each fiber $F(x)$ is $\gamma$-isotropic in $V(x)$; for $F$ of rank $2$, this is equivalent to requiring that the induced map $F \otimes F \to V^*\otimes L$ be zero.  If $F_1 \subset V$ is $\gamma$-isotropic, the bundle $E_{F_1} = \ker(V \to F_1^*\otimes V^* \otimes L)$ has rank $3$ and is isotropic for $B_\gamma$.  (If $u$ is a vector in a fiber $F_1(x)$, then $E_{F_1}(x) = E_u$, in the notation of \S\ref{introsec:isotropic}.)

Given a nondegenerate form $\gamma$ on $V$, there is a \define{$\gamma$-isotropic flag bundle} $\FFl_{\gamma}(V) \to X$, with fibers $\Fl_{\gamma}(V(x))$.  This comes with universal $\gamma$-isotropic subbundles $S_i$ and quotient bundles $Q_i$, as before.

\subsection{Chern class formulas} \label{introsec:formulas}

In the setup of \S\ref{introsec:bundles}, one has Schubert loci $\OOmega_w \subseteq \FFl_\gamma(V)$ indexed by the Weyl group.  There is an embedding of $W=W(G_2)$ in the symmetric group $S_7$ such that the permutation corresponding to $w\in W$ is determined by its first two values.  We identify $w$ with this pair of integers, so $w=w(1)\,w(2)$; see \S\ref{sec:weyl} for more on the Weyl group.  As in classical types, we set
\begin{eqnarray} \label{eqn:r-defn}
r_w(q,p) = \#\{i\leq q\,|\, w(i)\leq p\}.
\end{eqnarray}
Given a fixed $\gamma$-isotropic flag $F_\bullet$ on $X$, the Schubert loci are defined by
\begin{eqnarray*} 
\OOmega_w = \{x\in\FFl_\gamma(V) \,|\, \rk(F_p \to Q_q) \leq r_w(q,p) \text{ for } 1\leq p\leq 7, \, 1\leq q\leq 2 \}.
\end{eqnarray*}
These are locally trivial fiber bundles, whose fibers are Schubert varieties in $\Fl_\gamma(V(x))$.

The \define{$G_2$ divided difference operators} $\partial_s$ and $\partial_t$ act on $\Lambda{[x_1,x_2]}$, for any ring $\Lambda$, by
\begin{eqnarray}
\partial_s(f) &=& \frac{f(x_1,x_2) - f(x_2,x_1)}{x_1 - x_2} ; \label{eqn:divdiff-s}\\
\partial_t(f) &=& \frac{f(x_1,x_2) - f(x_1,x_1-x_2)}{-x_1 + 2x_2}. \label{eqn:divdiff-t}
\end{eqnarray}
If $w \in W$ has reduced word $w = s_{1}\cdot s_{2} \cdots s_{\ell}$ (where $s_i$ is the simple reflection $s$ or $t$), then define $\partial_w$ to be the composition $\partial_{s_{1}} \circ \cdots \circ \partial_{s_\ell}$.  This is independent of the choice of word; see \S\ref{sec:borel}.  (As mentioned in \S\ref{sec:weyl}, each $w \in W(G_2)$ has a unique reduced word, with the exception of $w_0$, so independence of choice is actually lack of choice in this case.)  These formulas also define operators on $H^*\FFl_\gamma(V)$.  (See \S\ref{sec:divdiff}.)

Let $V$ be a vector bundle of rank $7$ on $X$ equipped with a nondegenerate form $\gamma:\exterior^3 V \to k_X$, and assume $\det V$ is trivial.  Let $F_1 \subset F_2 \subset \cdots \subset V$ be a complete $\gamma$-isotropic flag in $V$.  Set $y_1 = c_1(F_1)$, $y_2 = c_1(F_2/F_1)$.  Let $\FFl_\gamma(V) \to X$ be the flag bundle, and set $x_1 = -c_1(S_1)$ and $x_2 = -c_1(S_2/S_1)$, where $S_1 \subset S_2 \subset V$ are the tautological bundles.

\begin{theorem} \label{thm:formula}
We have
\begin{eqnarray*}
[\OOmega_w] &=& \GP_w(x;y),
\end{eqnarray*}
where $\GP_w = \partial_{w_0\,w^{-1}} \GP_{w_0}$, and
\begin{eqnarray*}
\GP_{w_0}(x;y) &=& \frac{1}{2}( x_1^3 - 2\, x_1^2\, y_1 + x_1\, y_1^2 - x_1\, y_2^2 + x_1\, y_1\, y_2 - y_1^2\, y_2 + y_1\, y_2^2 ) \\
& & \times ( x_1^2 + x_1\, y_1 + y_1\, y_2 - y_2^2) (x_2 - x_1 - y_2).
\end{eqnarray*}
in $H^*(\FFl_\gamma(V),\Z)$.  (Here $w_0$ is the longest element of the Weyl group.)
\end{theorem}

\noindent
The proof is given in \S\ref{sec:divdiff}, along with a discussion of alternative formulas, including ones where $\gamma$ takes values in $M^{\otimes 3}$ for an arbitrary line bundle $M$.

\subsection{Degeneracy loci} \label{introsec:degloci}

Returning to the problem posed in the introduction, let $V$ be a rank $7$ vector bundle on a variety $X$, with nondegenerate form $\gamma$ and two (complete) $\gamma$-isotropic flags of subbundles $F_\bullet$ and $E_\bullet$.  The first flag, $F_\bullet$, allows us to define Schubert loci in the flag bundle $\FFl_\gamma(V)$ as in \S\ref{introsec:formulas}.  The second flag, $E_\bullet$, determines a section $s$ of $\FFl_\gamma(V) \to X$, and we define degeneracy loci as scheme-theoretic inverse images under $s$:
\begin{eqnarray*}
D_w = s^{-1}\OOmega_w \subset X.
\end{eqnarray*}
When $X$ is Cohen-Macaulay and $D_w$ has expected codimension (equal to the length of $w$; see \S\ref{sec:weyl}), we have 
\begin{eqnarray}
[D_w] = s^*[\OOmega_w] = \GP_w(x;y)
\end{eqnarray}
in $H^*X$, where $x_i = -c_1(E_i/E_{i-1})$ and $y_i = c_1(F_i/F_{i-1})$.  More generally, this polynomial defines a class supported on $D_w$, even without assumptions on the singularities of $X$ or the genericity of the flags $F_\bullet$ and $E_\bullet$; see \cite{flags} or \cite[App.\ A]{fp} for the intersection-theoretic details.

\subsection{Equivariant cohomology} \label{introsec:equivariant}

Now return to the case where $V$ is a $7$-dimensional vector space.  One can choose a basis $f_1,\ldots,f_7$ such that $F_i = \<f_1,\ldots,f_i\>$ forms a complete $\gamma$-isotropic flag in $V$, and let $T = (k^*)^2$ act on $V\isom k^7$ by
\begin{eqnarray*}
(z_1,z_2) \mapsto \diag(z_1, z_2, z_1 z_2^{-1}, 1, z_1^{-1} z_2, z_2^{-1}, z_1^{-1}).
\end{eqnarray*}
Write $t_1$ and $t_2$ for the corresponding weights.  Then $T$ preserves $\gamma$ and acts on $\Fl_{\gamma}(V)$.  The total equivariant Chern class of $V$ is $c^T(V) = (1-t_1^2)(1-t_2^2)(1-(t_1-t_2)^2)$, so we have
\begin{eqnarray*}
H_T^*(\Fl_{\gamma}(V),\Z[\textstyle{\frac{1}{2}}]) = \Z[\textstyle{\frac{1}{2}}][x_1,x_2,t_1,t_2]/(r_2,r_4,r_6),
\end{eqnarray*}
with the relations $r_{2i} = e_i(x_1^2,\, x_2^2,\, (x_1-x_2)^2) - e_i(t_1^2,\, t_2^2,\, (t_1-t_2)^2)$.  A presentation with $\Z$ coefficients can be deduced from Theorem \ref{thm:integral-presentation}; see Remark \ref{rmk:eq-presentation}.  

Theorem \ref{thm:formula} yields an equivariant Giambelli formula:
\begin{eqnarray*}
[\Omega_w]^T = \GP_w(x;t) \quad \text{ in } H_T^*\Fl_\gamma.
\end{eqnarray*}
In fact, this formula holds with integer coefficients: the Schubert classes form a basis for $H_T^*(\Fl_\gamma,\Z)$ over $\Z[t_1,t_2]$, so in particular there is no torsion, and $H_T^*(\Fl_\gamma,\Z)$ includes in $H_T^*(\Fl_\gamma,\Z[\frac{1}{2}])$.

\medskip

The equivariant geometry of $\Fl_\gamma$ is closely related to the degeneracy loci problem; we briefly describe the connection.  In the setup of \S\ref{introsec:degloci}, assume $V$ has trivial determinant and $\gamma$ has values in the trivial bundle, so the structure group is $G=G_2$.  The data of two $\gamma$-isotropic flags in $V$ gives a map to the classifying space $BB \times_{BG} BB$, where $B\subset G$ is a Borel subgroup, and there are universal degeneracy loci $\OOmega_w$ in this space.  On the other hand, there is an isomorphism $BB\times_{BG} BB \isom EB\times^B (G/B)$, carrying $\OOmega_w$ to $EB\times^B\Omega_w$.  Since $H_T^*(\Fl_\gamma) = H^*(EB\times^B (G/B))$, and $[\Omega_w]^T = [EB\times^B\Omega_w]$, a Giambelli formula for $[\Omega_w]^T$ is equivalent to a degeneracy locus formula for this situation.  One may then use equivariant localization to verify a given formula; this is essentially the approach taken in \cite{g2}.

\subsection{Other types}

It is reasonable to hope for a similar degeneracy locus story in some of the remaining exceptional types.  Groups of type $F_4$ and $E_6$ are closely related to \emph{Albert algebras}, and bundle versions of these algebras have been defined and studied over some one-dimensional bases \cite{pumpluen}.  Concrete realizations of the flag varieties have been given for types $F_4$ \cite{lm2}, $E_6$ \cite{im}, and $E_7$ \cite{gari-e7}.  Part of the challenge is to produce a complete flag from one of these realizations, and this seems to become more difficult as the dimension of the minimal irreducible representation increases with respect to the rank.

\section{Octonions and compatible forms} \label{sec:oct}

Any description of $G_2$ geometry is bound to be related to \emph{octonion algebras}, since the simple group of type $G_2$ may be realized as the automorphism group of an octonion algebra; see Proposition \ref{prop:auts} below.  For an entertaining and wide-ranging tour of the octonions (also known as the \emph{Cayley numbers} or \emph{octaves}), see \cite{baez}.

The basic linear-algebraic data can be defined as in \S\ref{sec:overview}, without reference to octonions, but the octonionic description is equivalent and sometimes more concrete.  In this section, we collect the basic facts about octonions that we will use, and establish their relationship with the notion of compatible forms introduced in \S\ref{introsec:compatible}.  Most of the statements hold over an arbitrary field, but we will continue to assume $k$ is algebraically closed of characteristic not $2$.

While studying holonomy groups of Riemannian manifolds, Bryant proved several related facts about octonions and representations of (real forms of) $G_2$.  In particular, he gives a way of producing a compatible bilinear form associated to a given trilinear form; we will use a version of this construction for forms on vector bundles.  See \cite{bryant} or \cite{harvey} for a discussion of the role of $G_2$ in differential geometry.

As far as I am aware, the results in \S\S\ref{sec:oct-forms}--\ref{sec:oct-bundles} have not appeared in the literature in this form, although related ideas about trilinear forms on a $7$-dimensional vector space can be found in \cite[\S2]{bryant}.

\subsection{Standard facts} \label{sec:oct-facts}

Here we list some well-known facts about composition algebras, referring to \cite[\S1]{sv} for proofs of any non-obvious assertions.

\begin{definition}
A \define{composition algebra} is a $k$-vector space $C$ with a nondegenerate quadratic norm $N: C\to k$ and an algebra structure $m:C\otimes C \to C$, with identity $e$, such that $N(uv) = N(u)N(v)$.  
\end{definition}

\noindent
Denote by $\beta'$ the symmetric bilinear form associated to $N$, defined by
\begin{eqnarray*}
\beta'(u,v) = N(u+v) - N(u) - N(v).
\end{eqnarray*}
(Notice that $\beta'(u,u) = 2N(u)$.)  Since $N(u) = N(eu) = N(e)N(u)$ for all $u \in C$, it follows that $N(e) = 1$ and $\beta'(e,e) = 2$.

The possible dimensions for $C$ are $1$, $2$, $4$, and $8$.  A composition algebra of dimension $4$ is called a \define{quaternion algebra}, and one of dimension $8$ is an \define{octonion algebra}; octonion algebras are neither associative nor commutative.  If there is a nonzero vector $u\in C$ with $N(u)=0$, then $C$ is \define{split}.  (Otherwise $C$ is a normed division algebra.)  Any two split composition algebras of the same dimension are isomorphic.  Over an algebraically closed field, $C$ is always split, so in this case there is only one composition algebra in each possible dimension, up to isomorphism.

Define \define{conjugation} on $C$ by $\bar{u} = \beta'(u,e)e - u$.  Every element $u\in C$ satisfies a quadratic \define{minimal equation} 
\begin{eqnarray} \label{eqn:minimal}
u^2 - \beta'(u,e)u + N(u) e = 0,
\end{eqnarray}
so
\begin{eqnarray}
u\bar{u} = \bar{u} u = N(u) e.
\end{eqnarray}
Write $V = e^{\perp} \subset C$ for the \define{imaginary subspace}.  For $u\in V$, $\bar{u} = -u$, so $u^2 = - N(u)e$, that is, $N(u) = -\frac{1}{2}\beta'(u^2,e)$.  For $u,v\in V$, we have
\begin{eqnarray}
\beta'(u,v)e &=& N(u+v)e - N(u)e - N(v)e  \nonumber \\
&=&  - uv - vu.
\end{eqnarray}
Although $C$ may not be associative, we always have $\bar{u}(uv) = (\bar{u}u)v = N(u)v$ and $(uv)\bar{v} = u(v\bar{v}) = N(v)u$ for any $u,v \in C$.  Also, for $u,v,w\in C$ we have
\begin{eqnarray} \label{eqn:orthogonal}
\beta'(uv,w) = \beta'(v,\bar{u}w) = \beta'(u,w\bar{v}).
\end{eqnarray}

A nonzero element $u\in C$ is a \define{zerodivisor} if there is a nonzero $v$ such that $uv=0$.  We have $0 = \bar{u}(uv) = (\bar{u}u)v = N(u) v$, so 
\begin{eqnarray} \label{eqn:zerodiv}
u \text{ is a zerodivisor iff } N(u)=0.
\end{eqnarray}

The relevance to $G_2$ geometry comes from the following:
\begin{proposition}[{\cite[\S2]{sv}}] \label{prop:auts}
Let $C$ be an octonion algebra over any field $k$.  Then the group $G = \Aut(C)$ of algebra automorphisms of $C$ is a simple group of type $G_2$, defined over $k$.  In fact, $G \subset SO(V,\beta) \subset SO(C,\beta')$, where $V = e^\perp$.  If $\Char(k)\neq 2$, $G$ acts irreducibly on $V$. \qedhere\qed
\end{proposition}

\subsection{Forms} \label{sec:oct-forms}

The algebra structure on $C$ corresponds to a trilinear form
\begin{eqnarray*}
\gamma': C \otimes C \otimes C \to k,
\end{eqnarray*}
using $\beta'$ to identify $C$ with $C^*$.  Specifically, we have
\begin{eqnarray} \label{eqn:recover-gamma}
\gamma'(u,v,w) = \beta'(uv,w).
\end{eqnarray}
Restricting $\gamma'$ to $V$, we get an alternating form which we will denote by $\gamma$.  (This follows from \eqref{eqn:orthogonal} and the fact that $\bar{u}=-u$ for $u\in V$.)  Also, $\beta'$ restricts to a nondegenerate form $\beta$ on $V$, defining an isomorphism $V \to V^*$.

The multiplication map $m:C\otimes C \to C$, with $C = k\oplus V$, is characterized by
\begin{eqnarray}
m(u,v) &=& -\frac{1}{2}\beta(u,v)e + \gamma(u,v,\cdot)^\dag \quad\text{ for }u,v\in V; \label{eqn:mult1} \\
m(u,e) &=& m(e,u) = u \quad\text{ for }u\in V; \label{eqn:mult2} \\
m(e,e) &=& e. \label{eqn:mult3}
\end{eqnarray}

Conversely, given a trilinear form $\gamma\in\exterior^3 V^*$ and a nondegenerate bilinear form $\beta\in\Sym^2 V^*$, extend $\beta$ orthogonally to $C=k\oplus V$ and define a multiplication $m$ according to formulas \eqref{eqn:mult1}--\eqref{eqn:mult3} above.

\begin{proposition} \label{prop:compatible-composition}
This multiplication makes $C$ into a composition algebra with norm $N(u) = \frac{1}{2}\beta'(u,u)$ if and only if $\gamma$ and $\beta$ are compatible, in the sense of Definition \ref{def:compatible}.
\end{proposition}

\begin{proof}
This is a simple computation:  For $u,v\in V$, we have
\begin{eqnarray*}
N(uv) &=& \frac{1}{2}\beta'(uv,uv) \\
&=& \frac{1}{2}\beta'(-\frac{1}{2}\beta(u,v)e, -\frac{1}{2}\beta(u,v)e) + \frac{1}{2}\beta'(\gamma(u,v,\cdot)^\dag, \gamma(u,v,\cdot)^\dag) \\
&=& \frac{1}{4}\beta(u,v)\beta(u,v) + \frac{1}{2} \gamma(u,v,\gamma(u,v,\cdot)^\dag),
\end{eqnarray*}
and
\begin{eqnarray*}
N(u) N(v) &=& \frac{1}{4}\beta(u,u)\beta(v,v). \hspace{200pt}\qedhere
\end{eqnarray*}
\end{proof}

\begin{remark}
Similar characterizations of octonionic multiplication have been given, usually in terms of a \emph{cross product} on $V$.  See \cite[\S2]{bryant} or \cite[\S6]{harvey}.
\end{remark}

\begin{lemma} \label{lemma:gamma-to-beta}
Suppose $\gamma$ and $\beta$ are compatible forms on $V$, defining a composition algebra structure on $C = k\oplus V$.  Then $L\subset V$ is $\gamma$-isotropic iff $uv=0$ in $C$ for all $u,v\in L$.  In particular, any $\gamma$-isotropic subspace is also $\beta$-isotropic.
\end{lemma}

\begin{proof}
Let $\gamma'$ and $\beta'$ be the forms corresponding to the algebra structure.  One implication is trivial: If $uv=0$ for all $u,v\in L$, then $\beta'(uv,\cdot) = \gamma'(u,v,\cdot)\equiv 0$ on $C$, so $\gamma(u,v,\cdot) \equiv 0$ on $V$ and $L$ is $\gamma$-isotropic.

Conversely, suppose $L$ is $\gamma$-isotropic.  First we show $L$ is $\beta$-isotropic.  Given any $u \in L$, choose a nonzero $v\in u^\perp \cap L$.  Since $L$ is $\gamma$-isotropic, $\gamma(u,v,\cdot)^\dag = 0$, so $u$ and $v$ are zerodivisors:
\begin{eqnarray*}
uv = -\frac{1}{2}\beta(u,v)\,e + \gamma(u,v,\cdot)^\dag = 0.
\end{eqnarray*}
Therefore $N(u) = N(v) = 0$, so $N$ and $\beta$ are zero on $L$.  By \eqref{eqn:mult1}, this also implies $uv = 0$ for all $u,v\in L$.
\end{proof}

Finally, it will be convenient to use certain bases for $C$ and $V$.  We need a well-known lemma:

\begin{lemma}[{\cite[(1.6.3)]{sv}}] \label{lemma:basic}
There are elements $a,b,c\in C$ such that
\begin{eqnarray*}
e, a, b, ab, c, ac, bc, (ab)c
\end{eqnarray*}
forms an orthogonal basis for $C$.  Such a triple is called a \define{basic triple} for $C$.
\end{lemma}

\noindent
In fact, given any $a\in V = e^{\perp}$ with $N(a) = 1$, we can choose $b$ and $c$ so that $a,b,c$ is an \emph{orthonormal} basic triple; similarly, if $a$ and $b$ are orthonormal vectors generating a quaternion subalgebra, we can find $c$ so that $a,b,c$ is an orthonormal basic triple.

If $a,b,c$ are an orthonormal basic triple, let $\{e_0 = e, e_1,\ldots, e_7\}$ be the corresponding basis (in the same order as in Lemma \ref{lemma:basic}).  This is a \define{standard orthonormal basis} for $C$.  With respect to the basis $\{e_1,\ldots,e_7\}$ for the imaginary octonions $V$, we have $\beta(e_p,e_q) = 2\,\delta_{pq}$, and
\begin{eqnarray} \label{eqn:e-gamma}
\gamma = 2\,( e_{123}^* + e_{257}^* - e_{167}^* - e_{145}^* - e_{246}^* - e_{347}^* - e_{356}^* ),
\end{eqnarray}
where $e_{pqr}^* = e_p^* \wedge e_q^* \wedge e_r^*$.  (Here $e_p^*$ is the map $e_q \mapsto \delta_{pq}$.)

\begin{remark}
Note that for $p>0$, $e_p^2 = -e$.  This standard orthonormal basis is analogous to the standard basis ``$1,i,j,k$'' for the quaternions.  Conventions for defining the octonionic product in terms of a standard basis vary widely in the literature, though --- Coxeter \cite[p.\ 562]{coxeter} calculates $480$ possible variations!  A choice of convention corresponds to a labelling and orientation of the Fano arrangement of $7$ points and $7$ lines; the one we use agrees with that of \cite[p.\ 363]{fh}.  (Coincidentally, our choice of $\gamma$ very nearly agrees with the one used in \cite[\S2]{bryant}: there the signs of $e_{347}^*$ and $e_{356}^*$ are positive, and the common factor of $2$ is absent.)
\end{remark}

We will most often use a different basis.  Define
\renewcommand{\arraystretch}{1.2}
\begin{eqnarray} \label{f-basis}
\begin{array}{rcl}
f_1 &=& \frac{1}{2}(e_1 + i\, e_2) \\
f_2 &=& \frac{1}{2}(e_5 + i\, e_6) \\
f_3 &=& \frac{1}{2}(e_4 + i\, e_7) \\
f_4 &=& i\,e_3 \\
f_5 &=& -\frac{1}{2}(e_4 - i\, e_7) \\
f_6 &=& -\frac{1}{2}(e_5 - i\, e_6) \\
f_7 &=& -\frac{1}{2}(e_1 - i\, e_2),
\end{array}
\end{eqnarray}
and call this the \define{standard $\gamma$-isotropic basis} for $V$.  (Here $i$ is a fixed square root of $-1$ in $k$.)  With respect to this basis, the bilinear form is given by 
\begin{eqnarray} \label{eqn:f-beta}
\begin{array}{rcl}
\beta(f_p,f_{8-q}) &=& - \delta_{pq}, \text{ for }p\neq 4 \text{ or } q\neq 4; \\
\beta(f_4,f_4) &=& -2.
\end{array}
\end{eqnarray}
The trilinear form is given by
\begin{eqnarray} \label{eqn:f-gamma}
\gamma = f_{147}^* + f_{246}^* + f_{345}^* - f_{237}^* - f_{156}^* .
\end{eqnarray}
(As above, $f_p^*$ denotes $f_q \mapsto \delta_{pq}$.)

\begin{example} \label{ex:oct-mult}
We can use the expression $\eqref{eqn:f-gamma}$ to compute the octonionic product $f_2\, f_3$.  By \eqref{eqn:mult1}--\eqref{eqn:mult3}, this is
\begin{eqnarray*}
f_2\, f_3 &=& -\frac{1}{2}\beta(f_2,f_3)\, e + \gamma(f_2,f_3,\cdot)^\dag \\
&=& \gamma(f_2,f_3,\cdot)^\dag.
\end{eqnarray*}
Since $\gamma(f_2,f_3,f_j) = -\delta_{7,j} = \beta(f_1, f_j)$, we see $\gamma(f_2,f_3,\cdot)^\dag = f_1$.  Therefore $f_2\, f_3 = f_1$.
\end{example}

We use computations in the $f$ basis to prove another characterization of nondegenerate forms.

\begin{lemma} \label{lemma:nondeg}
Let $\gamma:\exterior^3 V \to k$ be a trilinear form, and let $\beta_\gamma$ be a symmetric bilinear form defined as in \eqref{eqn:bryant-form} (for $\Char(k)\neq 3$), by composing
\begin{eqnarray*}
(u,v)\mapsto -\frac{1}{3}\gamma(u,\cdot,\cdot)\wedge\gamma(v,\cdot,\cdot)\wedge\gamma
\end{eqnarray*}
with an isomorphism $\exterior^7 V^* \isom k$.  Then $\gamma$ is nondegenerate if and only if $\beta_\gamma$ is nondegenerate.  (In fact, $\beta_\gamma$ is also defined if $\Char(k)=3$, and the same conclusion holds.)
\end{lemma}

\begin{proof}
Let $U \subset \exterior^3 V^*$ be the set of nondegenerate forms, and let $U'\subset\exterior^3 V^*$ be the set of forms $\gamma$ such that $\beta_\gamma$ is nondegenerate; we want to show $U=U'$.  (By Proposition \ref{prop:stabilizer}, $U$ is open and dense.)

First suppose $\gamma$ is nondegenerate.  Since $U$ is a $GL(V)$-orbit in $\exterior^3 V^*$, we may choose a basis $\{f_j\}$ so that $\gamma$ has the expression \eqref{eqn:f-gamma}.  Computing in this basis, and using $f_{1234567}^*$ to identify $\exterior^7 V^*$ with $k$, we find $\beta_\gamma = \beta$, i.e., $\beta_\gamma(f_p,f_{8-q}) = -\delta_{pq}$ for $p,q\neq 4$, and $\beta_\gamma(f_4,f_4) = -2$.  Indeed, we have
\begin{eqnarray*}
\gamma(f_1,\cdot,\cdot)\wedge\gamma(f_7,\cdot,\cdot)\wedge\gamma &=& (f_{47}^*-f_{56}^*)\wedge(f_{14}^*-f_{23}^*)\wedge\gamma \\
&=& 3 f_{1234567}^*.
\end{eqnarray*}
The others are similar.  In particular, with this choice of isomorphism $\exterior^7 V^* \isom k$, $\gamma$ and $\beta_\gamma$ are compatible forms.  (For an arbitrary choice of isomorphism, $\beta_\gamma$ is a scalar multiple of a compatible form.)

To see this works in characteristic $3$, one can avoid division by $3$.  Let $V_\Z$ be a rank $7$ free $\Z$-module, fix a basis $f_1,\ldots,f_7$, and let $\gamma_\Z:\exterior^3 V_\Z \to \Z$ be given by \eqref{eqn:f-gamma}.  The same computation shows that
\begin{eqnarray*}
\gamma_\Z(f_p,\cdot,\cdot)\wedge\gamma_\Z(f_{8-q},\cdot,\cdot)\wedge\gamma_\Z = 3\,\delta_{pq}\,f_{1234567}^*
\end{eqnarray*}
for $p,q\neq 4$, and
\begin{eqnarray*}
\gamma_\Z(f_4,\cdot,\cdot)\wedge\gamma_\Z(f_4,\cdot,\cdot)\wedge\gamma_\Z = 6\,\,f_{1234567}^*,
\end{eqnarray*}
so one can define $\beta_\gamma$ over $\Z$.  (For nondegeneracy, one still needs $\Char(k)\neq 2$ here.)

For the converse, note that the terms in the compatibility relation \eqref{eqn:compatible} make sense for all $\gamma$ in $U'$, since here $\gamma(u,v,\cdot)^\dag$ is well-defined.  We have seen that the relation holds on the dense open subset $U\subset U'$, so it must hold on all of $U'$.  Therefore every $\gamma$ in $U'$ has a compatible bilinear form, i.e., $\gamma$ is in $U$.
\end{proof}

The following two lemmas prove Proposition \ref{prop:3d}:

\begin{lemma}
If $u \in V$ is a nonzero isotropic vector, then
\begin{eqnarray*}
E_u &=& \{ v\in V \,|\, uv = 0\} \\
&=& \{ v\in V \,|\, \gamma(u,v,\cdot) \equiv 0 \}
\end{eqnarray*}
is a three-dimensional $\beta$-isotropic subspace.
\end{lemma}

\begin{proof}
By definition, $E_u$ consists of zero-divisors, so it is $\beta$-isotropic by \eqref{eqn:zerodiv}.  Since $\beta$ is nondegenerate on $V$, we know $\dim E_u \leq 3$.

In fact, it is enough to observe that $G = \Aut(C)$ acts transitively on the set of isotropic vectors (up to scalar); this follows from Proposition \ref{prop:g2flag}.  Thus for any $u$, we can find $g\in G$ such that $g\cdot u = \lambda f_1$ for some $\lambda \neq 0$.  Clearly $g\cdot E_u = E_{g\cdot u} = E_{f_1}$, and one checks that $f_1 f_2 = f_1 f_3 = 0$.
\end{proof}

\begin{lemma} \label{lemma:231}
Let $u\in V$ be a nonzero isotropic vector, and let $v,w\in E_u$ be such that $\{u,v,w\}$ is a basis.  Then $vw = \lambda u$ for some nonzero $\lambda \in k$.
\end{lemma}

\begin{proof}
First note that $vw = -wv$, since $-vw - wv = \beta(v,w)e =  0$.  If $\{u,v',w'\}$ is another basis, with $v' = a_1 u + a_2 v + a_3 w$ and $w' = b_1 u + b_2 v + b_3 w$, then $a_2 b_3 - a_3 b_2 \neq 0$, so
\begin{eqnarray*}
v' w' = (a_2 b_3) vw + (a_3 b_2) wv = (a_2 b_3 - a_3 b_2)vw
\end{eqnarray*}
is a nonzero multiple of $vw$.  Now it suffices to check this for the standard $\gamma$-isotropic basis, and indeed, we computed $f_2 f_3 = f_1$ in Example \ref{ex:oct-mult}.
\end{proof}

\begin{corollary} \label{cor:231}
Let $V = L_1 \oplus \cdots \oplus L_7$ be a splitting into one-dimensional subspaces such that $L_1$ is $\gamma$-isotropic, and $L_1 \oplus L_2 \oplus L_3 = E_u$ for a generator $u\in L_1$.  Then the map $V \otimes V \to V^* \isom V$ induced by $\gamma$ restricts to a $G$-equivariant isomorphism $L_2 \otimes L_3 \xrightarrow{\sim} L_1$. \qedhere\qed
\end{corollary}

Finally, the following lemma is verified by a straightforward computation:

\begin{lemma} \label{lemma:f-action}
Let $T=(k^*)^2$ act on $V$ via the matrix
\[
  \diag(z_1, z_2, z_1 z_2^{-1}, 1, z_1^{-1} z_2, z_2^{-1}, z_1^{-1})
\]
(in the $f$-basis).  Then $T$ preserves the forms $\beta$ and $\gamma$ of \eqref{eqn:f-beta} and \eqref{eqn:f-gamma}. \qedhere\qed
\end{lemma}

The corresponding weights for this torus action are $\{t_1, t_2, t_1-t_2, 0, t_2-t_1, -t_2, -t_1\}$.

\subsection{Octonion bundles} \label{sec:oct-bundles}

Let $X$ be a variety over $k$.  The notion of composition algebra can be globalized:
\begin{definition}
A \define{composition algebra bundle} over $X$ is a vector bundle $C \to X$, equipped with a nondegenerate quadratic norm $N:C \to k_X$, a multiplication $m:C\otimes C \to C$, and an identity section $e:k_X \to C$, such that $N$ respects composition.  (Equivalently, for each $x\in X$, the fiber $C(x)$ is a composition algebra over $k$.)
\end{definition}

Since $\Char(k)\neq 2$, there is a corresponding nondegenerate bilinear form $\beta'$ on $C$.  We will also allow composition algebras whose norm takes values in a line bundle $M^{\otimes 2}$; here the multiplication is $C\otimes C \xrightarrow{m} C\otimes M$, and the identity is $M\xrightarrow{e} C$.  Here a little care is required in the definition.  The composition $C \otimes M \xrightarrow{ id\otimes e} C \otimes C \xrightarrow{m} C \otimes M$ should be the identity, and the other composition ($m\circ(e\otimes id)$) should be the canonical isomorphism.  The compatibility between $m$ and $N$ is encoded in the commutativity of the following diagram:
\begin{diagram}
C\otimes C & & \rTo^m & & C \otimes M \\
& \rdTo_{N\otimes N} &    &  \ldTo_{N\otimes (N\circ e)} \\
&  &  M^{\otimes 4} .
\end{diagram}
The norm of $e$ is the quadratic map $M\to M^{\otimes 2}$ corresponding to $M^{\otimes 2} \xrightarrow{\beta'} M^{\otimes 2}$.  Replacing $C$ with $\tilde{C} = C\otimes M^*$, one obtains a composition algebra whose norm takes values in the trivial bundle.

Many of the properties of composition algebras discussed above have straightforward generalizations to bundles; we mention a few without giving proofs.

Using $\beta'$ to identify $C$ with $C^*\otimes M^{\otimes 2}$, the multiplication map corresponds to a trilinear form $\gamma':C\otimes C \otimes C \to M^{\otimes 3}$.  The \define{imaginary subbundle} $V$ is the orthogonal complement to $e$ in $C$, so $C = M \oplus V$.  The bilinear form $\beta'$ restricts to a nondegenerate form $\beta$ on $V$, and $\gamma'$ restricts to an alternating form $\gamma:\exterior^3 V \to M^{\otimes 3}$.  As before, the multiplication on $M\oplus V$ can be recovered from the forms $\beta$ and $\gamma$ on $V$, and there is an analogue of Proposition \ref{prop:compatible-composition}.

The analogues of Proposition \ref{prop:3d} and Corollary \ref{cor:231} can be proved using octonion bundles and reducing to the local case:

\begin{proposition} \label{prop:3d-bundle}
Let $\gamma:\exterior^3 V \to M^{\otimes 3}$ and $\beta:\Sym^2 V \to M^{\otimes 2}$ be (locally) compatible forms.  Let $F_1 \subset V$ be a $\gamma$-isotropic line bundle, and let $\phi:V \to F_1^* \otimes V^* \otimes M^{\otimes 3}$ be the map defined by $\gamma$.  Then the bundle
\begin{eqnarray*}
E_{F_1} = \ker(\phi)
\end{eqnarray*}
has rank $3$ and is $\beta$-isotropic. \qedhere\qed
\end{proposition}

\begin{proposition} \label{prop:231-bundle}
Let $V$ be as in Proposition \ref{prop:3d-bundle}, and suppose there is a splitting $V = L_1 \oplus \cdots \oplus L_7$ into line bundles such that $L_1$ is $\gamma$-isotropic, and $L_1 \oplus L_2 \oplus L_3 = E_{L_1}$.  Then the map $V \otimes V \to V^* \isom V\otimes M$ induced by $\gamma$ and $\beta$ restricts to an isomorphism $L_2 \otimes L_3 \xrightarrow{\sim} L_1 \otimes M$. \qedhere\qed
\end{proposition}

\begin{remark}
Composition algebras may defined over an arbitrary base scheme $X$; in fact, as with Azumaya algebras, one is mainly interested in cases where $X$ is defined over a non-algebraically closed field or a Dedekind ring.  Petersson has classified such composition algebra bundles in the case where $X$ is a curve of genus zero \cite{petersson}.  Since then, some work has been done over other one-dimensional bases, but the theory remains largely undeveloped.
\end{remark}

\section{Topology of $G_2$ flags} \label{sec:topology}

There are two ``$\gamma$-isotropic Grassmannians'' parametrizing $\gamma$-isotropic subspaces of dimensions $1$ or $2$, which we write as $\QQ$ or $\GG$, respectively; thus $\Fl_\gamma$ embeds in $\QQ\times\GG$.  Since $\gamma$-isotropic vectors are just those $v$ such that $\beta(v,v)=0$, $\QQ$ is the smooth $5$-dimensional quadric hypersurface in $\P(V)$.  

\begin{proposition} \label{prop:gamma-flags}
The $\gamma$-isotropic flag variety is a smooth, $6$-dimensional projective variety.  Moreover, both projections $\Fl_\gamma \to \QQ$ and $\Fl_\gamma \to \GG$ are $\P^1$-bundles.
\end{proposition}

\begin{proof}
The quadric $\QQ$ comes with a tautological line bundle $S_1\subset V_\QQ$.  By Proposition \ref{prop:3d}, the form $\gamma$ also equips $\QQ$ with a rank-$3$ bundle $S_3 \subset V_\QQ$, with fiber $S_3({[u]}) = E_u$, the space swept out by all $\gamma$-isotropic $2$-spaces containing $u$.  Thus $S_1\subset S_3$, and from the definitions we have $\Fl_\gamma(V) = \P(S_3/S_1) \to \QQ$.  (We use the convention that $\P(E)$ parametrizes lines in the vector bundle $E$.)

Similarly, if $S_2$ is the tautological bundle on $\GG$, we have $\Fl_\gamma(V) = \P(S_2) \to \GG$.  This also shows that $\GG$ is smooth of dimension $5$.
\end{proof}

\begin{remark}
The definition of $\Fl_\gamma(V)$ can be reformulated as follows.  Let $\Fl=\Fl(1,2;V)$ be the two-step partial flag variety.  The nondegenerate form $\gamma$ is also a section of the trivial vector bundle $\exterior^3 V^*$ on $\Fl$.  By restriction it gives a section of the rank $5$ vector bundle $\exterior^2 S_2^* \otimes Q_5^*$, where $S_1 \subset S_2 \subset V$ is the tautological flag on $\Fl$ and $Q_5 = V/S_2$.  Then $\Fl_\gamma \subset \Fl$ is defined by the vanishing of this section.
\end{remark}

\begin{remark}
Projectively, $\Fl_\gamma$ parametrizes data $(p\in\ell)$, where $\ell$ is a $\gamma$-isotropic line in $\QQ$, and $p\in\ell$ is a point.  Thus Proposition \ref{prop:3d} says that the union of such $\ell$ through a fixed $p$ is a $\P^2$ in $\QQ$, and conversely, given such a $\P^2$ one can recover $p$ (as the intersection of any two $\gamma$-isotropic lines in the $\P^2$).
\end{remark}

\subsection{Fixed points} \label{subsec:fixed}

Let $\{f_1,f_2,\ldots,f_7\}$ be the standard $\gamma$-isotropic basis for $V$, and let $T = (k^*)^2$ act as in Lemma \ref{lemma:f-action}, via the matrix $\diag(z_1, z_2, z_1 z_2^{-1}, 1, z_1^{-1} z_2, z_2^{-1}, z_1^{-1})$.  Write $e(i\,j)$ for the two-step flag $\< f_i \>\; \subset \; \< f_i, f_j \>$.

\begin{proposition} \label{prop:fixed-pts}
This action of $T$ defines an action on $\Fl_\gamma(V)$, with $12$ fixed points:
\begin{eqnarray*}
e(1\,2),\; e(1\,3),\; e(2\,1),\; e(2\,5),\; e(3\,1),\; e(3\,6),\; \\ e(5\,2),\; e(5\,7),\; e(6\,3),\; e(6\,7),\; e(7\,5),\; e(7\,6).
\end{eqnarray*}
\end{proposition}

\begin{proof}
Since $T$ preserves $\beta$, it acts on $\QQ$, fixing the $6$ points ${[f_1]}$, ${[f_2]}$, ${[f_3]}$, ${[f_5]}$, ${[f_6]}$, ${[f_7]}$.  Since $T$ preserves $\gamma$, it acts on $\Fl_\gamma$, and the projection $\Fl_\gamma \to \QQ$ is $T$-equivariant.  The $T$-fixed points of $\Fl_\gamma$ lie in the fibers over the fixed points of $\QQ$.  Since each of these $6$ fibers is a $\P^1$ with nontrivial $T$-action, there must be $2\cdot 6 = 12$ fixed points.

To see the fixed points are as claimed, note that the bundle $S_3$ on $\QQ$ is equivariant, and the fibers $S_3(x) = E_x$ at each of the fixed points are as follows:
\begin{eqnarray*}
E_{f_1} &=& \< f_1, f_2, f_3 \> \\
E_{f_2} &=& \< f_2, f_1, f_5 \> \\
E_{f_3} &=& \< f_3, f_1, f_6 \> \\
E_{f_5} &=& \< f_5, f_2, f_7 \> \\
E_{f_6} &=& \< f_6, f_3, f_7 \> \\
E_{f_7} &=& \< f_7, f_5, f_6 \> .
\end{eqnarray*}
Indeed, one simply checks that in each triple, the (octonionic) product of the first vector with either the second or the third is zero.  (Alternatively, one can compute directly using the form \eqref{eqn:f-gamma}.)  Now the $T$-fixed lines in $S_3({[f_i]})/S_1({[f_i]})$ are ${[f_j]}$, where $f_j$ is the second or third vector in the triple beginning with $f_i$.  Thus the $12$ points are $e(i\,j)$, where $f_i$ is the first vector and $f_j$ is the second or third vector in one of the above triples.
\end{proof}

\medskip

In general, the $T$-fixed points of a flag variety are indexed by the corresponding Weyl group $W$, which for type $G_2$ is the dihedral group with $12$ elements.  We will write elements as $w = w(1)\, w(2)$, for $w(1)$ and $w(2)$ such that $e(w(1)\,w(2))$ is a $T$-fixed point, as in Proposition \ref{prop:fixed-pts}.  We fix two simple reflections generating $W$, $s=2\,1$ and $t=1\,3$.  See \S\ref{sec:weyl} for more details on the Weyl group and its embedding in $S_7$.

\subsection{Schubert varieties} \label{subsec:schubert}

Fix a (complete) $\gamma$-isotropic flag $F_{\bullet}$ in $V$.  Each $T$-fixed point is the center of a \define{Schubert cell}, which is defined by
\begin{eqnarray*}
X_{w}^o = \{ E_{\bullet} \in \Fl_{\gamma} \,|\, \dim(F_p \cap E_q) = r_w(q,p) \text{ for } 1\leq q\leq 2,\, 1\leq p\leq 7 \},
\end{eqnarray*}
where $r_w(q,p) = \#\{i\leq q \,|\, w(i) \leq p\}$, just as in the classical types.  Also as in classical types, these can be parametrized by matrices, where $E_i$ is the span of the first $i$ rows.  For example, the big cell is
\begin{eqnarray*}
X^o_{7\,6} = \left(\begin{array}{ccccccc}
X & a & b & c & d & e & 1 \\
Y & Z & S & T & f & 1 & 0\end{array}\right) \isom \A^6,
\end{eqnarray*}
where lowercase variables are free, and $X,Y,Z,S,T$ are given by
\begin{eqnarray*}
X &=& -ae - bd - c^2 \\
Y &=& -a - bf + cd - cef \\
Z &=& -cf - d^2 + def \\
S &=& c + de - e^2 f\\
T &=& -d + ef .
\end{eqnarray*}
(These equations can be obtained by octonionic multiplication; considering the two row vectors as imaginary octonions, the condition is that their product be zero.  In fact, $X,Y,Z$ are already determined by $\beta$-isotropicity.)  Parametrizations of the other $11$ cells are given in \cite[Appendix D.1]{thesis}.

The \define{Schubert varieties} $X_w$ are the closures of the Schubert cells; equivalently,
\begin{eqnarray*}
X_{w} = \{ E_{\bullet} \in \Fl_{\gamma} \,|\, \dim(F_p \cap E_q) \geq r_w(q,p) \text{ for } 1\leq q\leq 2, \, 1\leq p\leq 7 \}.
\end{eqnarray*}
From the parametrizations of cells, we see $\dim X_w = \ell(w)$.  To get Schubert varieties with codimension $\ell(w)$, define
\begin{eqnarray*}
\Omega_w = X_{w\,w_0}.
\end{eqnarray*}
These can also be described using the tautological quotient bundles:
\begin{eqnarray*}
\Omega_w = \{ x\in \Fl_\gamma \,|\, \rk(F_p(x) \to Q_q(x)) \leq r_w(q,p) \}.
\end{eqnarray*}

Schubert varieties in $\QQ$ and $\GG$ are defined by the same conditions.  (Note that $w$ and $w\,s$ define the same varieties in $\GG$, and $w$ and $w\,t$ define the same variety in $\QQ$.  Write $\bar{w}$ for the corresponding equivalence class.)  With the exception of $X_{1\,2}$, all Schubert varieties in $\Fl_\gamma$ are inverse images of Schubert varieties in $\QQ$ or $\GG$: 

\begin{proposition} \label{prop:schubert-inverse}
Let $p:\Fl_\gamma\to\QQ$ and $q:\Fl_\gamma\to\GG$ be the projections.  Then $X_w = p^{-1}X_{\bar{w}}$ if $w(1)<w(2)$ (except when $w=1\,2$), and $X_w = q^{-1}X_{\bar{w}}$ if $w(1)>w(2)$.
\end{proposition}

The proof is immediate from the definitions.  For instance, $X_{\bar{tst}} = X_{\bar{36}}$ is a $\P^2$ in $\QQ$: it parametrizes all $1$-dimensional subspaces of a fixed isotropic $3$-space.  Its inverse image in $\Fl_\gamma$ is $p^{-1}X_{\bar{tst}} = X_{tst} = \Omega_{sts}$.

\section{Cohomology of flag bundles} \label{sec:cohomology}

%
\subsection{Compatible forms on bundles} \label{subsec:bundle-forms}

Let $V$ be a rank $7$ vector bundle on a variety $X$, equipped with a nondegenerate form $\gamma:\exterior^3 V \to L$, and let $B_\gamma:\Sym^2 V \to \det V^* \otimes L^{\otimes 3}$ be the Bryant form (\S\ref{introsec:compatible}, \S\ref{introsec:bundles}).  Assume there is a line bundle $M$ such that
\begin{eqnarray} \label{eqn:sqrt}
\det V^* \otimes L^{\otimes 3} \isom M^{\otimes 2}.
\end{eqnarray}
(For example, this holds if $V$ has a maximal $B_\gamma$-isotropic subbundle $F$, for then we can take $M = F^\perp/F$.  There exist Zariski-locally trivial bundles $V$ without this property, though --- see \cite[p.\ 293]{eg}.)

\begin{lemma} \label{lemma:lbs}
In this setup, $L \isom M^{\otimes 3} \otimes T$, for some line bundle $T$ such that $T^{\otimes 3}$ is trivial.  If $L$ has a cube root, then $T$ is trivial and $M\isom \det V \otimes (L^*)^{\otimes 2}$.
\end{lemma}

\noindent
The proof is straightforward; see \cite[Lemma 3.2.1]{thesis} for details.

From now on, we will assume $V$ has a maximal $B_\gamma$-isotropic subbundle $F = F_3 \subset V$.  We also assume $L$ has a cube root on $X$, so $L \isom M^{\otimes 3}$.  (By a theorem of Totaro, one can always assume this so long as $3$-torsion is ignored in Chow groups (or cohomology); see \cite{clgp}.  In the case at hand, Lemma \ref{lemma:lbs} gives a direct reason.)

\subsection{A splitting principle}

For the next three subsections, we assume the line bundle $M$ is trivial; this implies $\det V$ is also trivial.  The case for general $M$ will be described in \S\ref{subsec:twist}.

In this context, the relevant version of the splitting principle is the following:
\begin{lemma} \label{lemma:split}
Assume $V$ is equipped with a nondegenerate trilinear form $\gamma:\exterior^3 V \to k_X$.  There is a map $f: Z \to X$ such that $f^*: H^*X \to H^*Z$ is injective, and $f^*V \isom L_1 \oplus L_2 \oplus \cdots \oplus L_7$, with $E_i = L_1 \oplus \cdots \oplus L_i$ forming a complete $\gamma$-isotropic flag in $f^*V$.
\end{lemma}

\noindent
The proof is given in \cite[Lemma 3.2.2]{thesis}.

Given such a splitting, we can use $\beta$ to identify $L_{8-i}$ with $L_i^*$, and Proposition \ref{prop:231-bundle} implies $L_3 \isom L_1 \otimes L_2^*$.  Thus
\begin{eqnarray} \label{eqn:v-splitting}
V \isom L_1 \oplus L_2 \oplus (L_1 \otimes L_2^*) \oplus k_X \oplus (L_1^* \otimes L_2) \oplus L_2^* \oplus L_1^*.
\end{eqnarray}
Alternatively, using only a $\gamma$-isotropic flag of subbundles $F_1 \subset F_2 \subset V$, we have
\begin{eqnarray*}
V \isom F_2 \oplus (F_1\otimes(F_2/F_1)^*) \oplus k_X \oplus (F_1^* \otimes (F_2/F_1)) \oplus F_2^*.
\end{eqnarray*}

Since $V$ is recovered from the data of $L_1$ and $L_2$, the universal base for $V$ (with the assumed restrictions) is $BGL_1\times BGL_1$.  

\subsection{Chern classes}

We continue to assume the line bundle $M$ is trivial, and let $F_1 \subset F_2 \subset F_3 \subset V$ be a $\gamma$-isotropic flag in $V$.  It follows from \eqref{eqn:v-splitting} that
\begin{eqnarray} \label{eqn:roots}
c(V) = (1 - y_1^2)(1 - y_2^2)(1 - (y_1 - y_2)^2) ,
\end{eqnarray}
where $y_i = c_1(L_i)$, and also that
\begin{eqnarray*}
c_1(F_3) = 2\,c_1(F_1).
\end{eqnarray*}
Let $\QQ(V) \to X$ be the quadric bundle, with its tautological bundles $S_1\subset S_3 \subset V$.  Set $x_1 = -c_1(S_1)$ and $\alpha = [\P(F_3)]$ in $H^*\QQ(V)$.  The classes $1,\, x_1,\, x_1^2,\, \alpha,\, x_1\,\alpha,\, x_1^2\,\alpha$ form a basis for $H^*\QQ(V)$ over $H^*X$; see Appendix \ref{ch:chow}.

\begin{lemma} \label{lemma:chern}
We have
\begin{eqnarray*}
c_1(S_3) &=& -2\,x_1,  \quad\text{ and} \\
c_2(S_3) &=& 2\,x_1^2 + c_2(F_3) - 2\,c_1(F_1)^2.
\end{eqnarray*}
\end{lemma}

\begin{proof}
The expression for $c_1(S_3)$ follows from \eqref{eqn:v-splitting}.  We have $V/F_3^\perp \isom F_3^*$ and $V/S_3^\perp \isom S_3^*$, so $c(V) = c(F_3)\cdot c(F_3^*) = c(S_3)\cdot c(S_3^*)$.  In particular,
\begin{eqnarray*}
c_2(V) &=& 2\,c_2(F_3) - c_1(F_3)^2 = 2\,c_2(S_3) - c_1(S_3)^2 \\
&=& 2(c_2(F_3) - 2\,c_1(F_1)^2) = 2(c_2(S_3) - 2\,x_1^2).
\end{eqnarray*}
Up to $2$-torsion, then, the formula for $c_2(S_3)$ holds.  Since the classifying space for this setup is $BGL_1 \times BGL_1$, and there is no torsion in its cohomology, it follows that the formula also holds with integer coefficients.
\end{proof}

\subsection{Presentations} \label{subsec:presentations}

Using the fact that $\FFl_\gamma(V)$ is a $\P^1$-bundle over a quadric bundle, we can give a presentation of its integral cohomology.  First recall the presentation for $H^*\QQ(V)$ (Theorem \ref{thm:chow}).  We continue to assume $M$ is trivial, and hence also $\det V$.  Fix $F_1 \subset F_3 \subset V$ as before, and let $S_1 \subset S_3 \subset V$ be the tautological bundles on $\QQ(V)$.  Let $x_1 = -c_1(S_1)$ and $\alpha = [\P(F_3)]$ in $H^*\QQ(V)$.  Then
\begin{eqnarray*}
H^*(\QQ(V),\Z) = (H^*X)[x_1,\alpha]/I,
\end{eqnarray*}
where $I$ is generated by
\begin{eqnarray*}
2\alpha &=& x_1^3 - c_1(F_3)\,x_1^{2} + c_2(F_3)\,x_1 - c_3(F_3) , \\
\alpha^2 &=& (c_3(V/F_3) + c_{1}(V/F_3)\,x_1^2)\,\alpha.
\end{eqnarray*}

\begin{theorem} \label{thm:integral-presentation}
With notation as above, we have $\FFl_\gamma(V) = \P(S_3/S_1) \to \QQ(V) \to X$.  Let $x_2 = - c_1(S_2/S_1)$ be the hyperplane class for this $\P^1$-bundle.  Then
\begin{eqnarray*}
H^*(\FFl_\gamma(V),\Z) = (H^*X)[x_1,x_2,\alpha]/J,
\end{eqnarray*}
where $J$ is generated by the three relations
\begin{eqnarray}
2\alpha &=& x_1^3 - c_1(F_3)\,x_1^{2} + c_2(F_3)\,x_1 - c_3(F_3) , \\
\alpha^2 &=& (c_3(V/F_3) + c_{1}(V/F_3)\,x_1^2)\,\alpha , \\
x_1^2 + x_2^2 - x_1 x_2 &=& 2\,c_1(F_1)^2 - c_2(F_3).
\end{eqnarray}
\end{theorem}

In fact, $\alpha$ is the Schubert class $[\OOmega_{sts}]$, defined in \S\ref{sec:divdiff} below.

\begin{proof}
Since $\FFl_\gamma(V) = \P(S_3/S_1) \to \QQ(V)$, we have
\begin{eqnarray*}
H^*\FFl_\gamma = (H^*\QQ)[x_2]/( x_2^2 + c_1(S_3/S_1)\,x_2 + c_2(S_3/S_1) ).
\end{eqnarray*}
One easily checks $c_1(S_3/S_1) = -x_1$, and
\begin{eqnarray*}
c_2(S_3/S_1) = c_2(S_3) - x_1^2 = x_1^2 + c_2(F_3) - 2\,c_1(F_1)^2
\end{eqnarray*}
by Lemma \ref{lemma:chern}.  This gives the third relation, and the first two relations come from the relations on $H^*\QQ$.

Finally, it is not hard to see that the $12$ elements 
\begin{eqnarray*}
1,\, x_1, x_1^2,\, \alpha,\, x_1\,\alpha,\, x_1^2\,\alpha,\, x_2,\, x_1\,x_2,\, x_1^2\,x_2,\, x_2\,\alpha,\, x_1\,x_2\,\alpha,\, x_1^2\, x_2\, \alpha
\end{eqnarray*}
form a basis for the ring on the RHS over $H^*X$, and we know they form a basis for $H^*\FFl_\gamma$ over $H^*X$.
\end{proof}

\begin{remark} \label{rmk:eq-presentation}
To obtain a presentation for $H_T^*(\Fl_\gamma,\Z)$, set $\alpha=[\Omega_{sts}]^T$, $x_i = -c_1^T(S_i/S_{i-1})$, $c_i(F_3) = (-1)^i c_i(V/F_3) = e_i(t_1,t_2,t_1-t_2)$, and $c_1(F_1)=t_1$.
\end{remark}

If we take coefficients in $\Z[\frac{1}{2}]$, the cohomology ring has a simpler presentation similar to that for classical groups:
\begin{proposition} \label{prop:bundle-presentation}
Suppose $V$ has a splitting as in \eqref{eqn:v-splitting}, with $M$ trivial.  Let $\Lambda = H^*X$.  Then $H^*(\FFl_{\gamma}(V), \Z[\frac{1}{2}]) \isom \Lambda{[x_1,x_2]}/(r_2,r_4,r_6)$, where
\begin{eqnarray*}
r_{2i} &=& e_i(x_1^2, x_2^2, (x_1 - x_2)^2 ) - e_i(y_1^2, y_2^2, (y_1 - y_2)^2).
\end{eqnarray*}
\end{proposition}

\begin{proof}
The relations must hold, by \eqref{eqn:roots}.  Monomials in $x_1$ and $x_2$ are global classes on $\FFl_\gamma$ that restrict to give a basis for the cohomology of each fiber, so the claim follows from the Leray--Hirsch theorem.
\end{proof}

Taking $X$ to be a point, these presentations specialize to give well-known presentations of $H^*\Fl_\gamma$ (cf.\ \cite{bs}):
\begin{corollary}
Let $\Fl_\gamma$ be the $\gamma$-isotropic flag variety, and let $p:\Fl_\gamma\to\QQ$ be the projection to the quadric.  Set $\alpha = [\Omega_{sts}] \in H^*(\Fl_\gamma,\Z)$.  Then we have
\begin{eqnarray*}
H^*(\Fl_\gamma,\Z) &=& \Z[x_1,x_2,\alpha]/(x_1^2 + x_2^2 - x_1 x_2,\, 2\,\alpha - x_1^3,\, \alpha^2),
\end{eqnarray*}
and
\begin{eqnarray*}
H^*(\Fl_\gamma,\Z[\textstyle{\frac{1}{2}}]) &=& \Z[\textstyle{\frac{1}{2}}][x_1,x_2]/( e_i( x_1^2, x_2^2, (x_1-x_2)^2 ) )_{i=1,2,3} \\
&=& \Z[\textstyle{\frac{1}{2}}][x_1,x_2]/( x_1^2 + x_2^2 - x_1 x_2, \, x_1^6 ).
\end{eqnarray*}
\end{corollary}

\subsection{Twisting} \label{subsec:twist}

Now we allow $\gamma$ to take values in $L \isom M^{\otimes 3}$ for an arbitrary line bundle $M$ on $X$, so $\det V \isom M^{\otimes 7}$ and the corresponding bilinear form has values in $M^{\otimes 2}$.

The splitting principle (Lemma \ref{lemma:split}) holds as stated for $\gamma:\exterior^3 V \to M^{\otimes 3}$.  The compatible bilinear form $\beta$ now identifies $L_{8-i}$ with $L_i^*\otimes M^{\otimes 2}$, and we have $L_3 \isom L_1 \otimes L_2^* \otimes M$.  Thus
\begin{eqnarray} \label{eqn:v-splitting-m}
V &\isom& L_1 \oplus L_2 \oplus (L_1 \otimes L_2^* \otimes M) \oplus M \nonumber \\
& & \oplus (L_1^* \otimes L_2 \otimes M) \oplus (L_2^* \otimes M^{\otimes 2}) \oplus (L_1^* \otimes M^{\otimes 2}).
\end{eqnarray}
Since $V$ is recovered from the data of $L_1$, $L_2$, and $M$, the universal base for $V$ is $(BGL_1)^3$.  This space has no torsion in cohomology; it follows that we may deduce integral formulas using rational coefficients.

As described in \cite{clgp}, this situation reduces to the case where $L$ is trivial.  Let $\tilde{V} = V\otimes M^*$, so $\gamma:\exterior^3 V \to L$ determines a form $\tilde{\gamma}:\exterior^3 \tilde{V} \to  k_X$.  If $V = L_1\oplus \cdots \oplus L_7$ is a $\gamma$-isotropic splitting as in Lemma \ref{lemma:split}, we have $\tilde{V} = \tilde{L}_1 \oplus \cdots \oplus \tilde{L}_7$, where $\tilde{L}_i = L_i\otimes M^*$.  Thus
\begin{eqnarray*}
c(\tilde{V}) = (1 - \tilde{y}_1^2)(1 - \tilde{y}_2^2)(1 - \tilde{y}_3^2),
\end{eqnarray*}
where $v = c_1(M)$, $\tilde{y}_i = y_i - v$, so $\tilde{y}_3 = \tilde{y}_1 - \tilde{y}_2 = y_1 - y_2$.  Note that $y_1 - y_2 = y_3 - v$, since using $\gamma$ and $\beta$ there is an isomorphism $L_2 \otimes L_3 \isom L_1 \otimes M$.

A rank $2$ subbundle $E\subset V$ is $\gamma$-isotropic if and only if $\tilde{E} = E\otimes M^* \subset \tilde{V}$ is $\tilde\gamma$-isotropic (a map is zero iff it is zero after twisting by a line bundle), so we have an isomorphism $\FFl_\gamma(V) \isom \FFl_{\tilde\gamma}(\tilde{V})$, and the tautological subbundles are related by $\tilde{S}_i = S_i \otimes M^*$.  Therefore $\tilde{x}_i = -c_1(\tilde{S}_i/\tilde{S}_{i-1}) = x_i + v$.  The presentation for $H^*\FFl_\gamma(V)$ is obtained from Proposition \ref{prop:bundle-presentation} by replacing $y_i$ with $y_i - v$ and $x_i$ with $x_i + v$.

\section{Divided difference operators and Chern class formulas} \label{sec:divdiff}

For now, assume $\gamma$ takes values in the trivial bundle.  Given $V \to X$ with a (complete) $\gamma$-isotropic flag of subbundles $F_\bullet$, Schubert loci $\OOmega_w \subset \FFl_\gamma(V)$ are defined by rank conditions as in \S\ref{introsec:formulas}.  (These are the same conditions as in \S\ref{subsec:schubert} when $X$ is a point.)  
As usual, there are two steps to producing formulas for these Schubert loci: first find a formula for the most degenerate locus (the case $w=w_0$), and then apply divided difference operators to obtain formulas for all $w\leq w_0$.  Theorem \ref{thm:mytopclass} and Lemma \ref{lemma:divdiff} prove Theorem \ref{thm:formula}.

\begin{theorem} \label{thm:mytopclass}
Assume $M$ is trivial, and let $F_1 \subset F_2 \subset F_3 \subset V$ be a $\gamma$-isotropic flag.  Then $[\OOmega_{w_0}] \in H^*\FFl_\gamma(V)$ is given by
\begin{eqnarray*}
[\OOmega_{w_0}] &=& \frac{1}{2}( x_1^3 - c_1(F_3)\, x_1^2 + c_2(F_3)\, x_1 - c_3(F_3) ) \\
& & \times ( x_1^2 + c_1(F_1)\, x_1 + c_2(F_3) - c_1(F_1)^2 )( x_2 - x_1 - c_1(F_3/F_1) ).
\end{eqnarray*}
\end{theorem}

Setting $y_1 = c_1(F_1)$ and $y_2 = c_1(F_2/F_1)$, we have $c(F_3) = (1 + y_1)(1 + y_2)(1 + y_1 - y_2)$, so this formula becomes $[\OOmega_{w_0}] = \GP_{w_0}(x;y)$, where
\begin{eqnarray*}
\GP_{w_0}(x;y) &=& \frac{1}{2}( x_1^3 - 2\, x_1^2\, y_1 + x_1\, y_1^2 - x_1\, y_2^2 + x_1\, y_1\, y_2 - y_1^2\, y_2 + y_1\, y_2^2 ) \\
& & \times ( x_1^2 + x_1\, y_1 + y_1\, y_2 - y_2^2) (x_2 - x_1 - y_2).
\end{eqnarray*}

\begin{proof}
Let $p:\FFl_\gamma = \P(S_3/S_1) \to \QQ$ be the projection.  The locus where $S_1 = F_1$ is $p^{-1}\P(F_1)$, so its class is $p^*[\P(F_1)]$.  On $\P(F_1) \subset \QQ$, we have $S_1 = F_1$ and $S_3 = F_3$; thus on $p^{-1}\P(F_1)$, the locus where $S_2 = F_2$ is defined by the vanishing of the composed map $F_2/F_1 = F_2/S_1 \to S_3/S_1 \to S_3/S_2$.  This class is given by $c_1( (F_2/F_1)^* \otimes S_3/S_2 ) = x_2 - x_1 - c_1(F_2/F_1)$, so pushing forward by the inclusion $p^{-1}\P(F_1) \hookrightarrow \FFl_\gamma$, we have
\begin{eqnarray*}
[\OOmega_{w_0}] = p^*[\P(F_1)] \cdot (x_2 - x_1 - c_1(F_2/F_1)).
\end{eqnarray*}

To determine $[\P(F_1)]$ in $H^*\QQ$, we first find the class in $H^*\P(F_3)$ and then push forward.  By \cite[Ex.\ 3.2.17]{it}, this is $x_1^2 + c_1(F_3/F_1)\, x_1 + c_2(F_3/F_1)$, and pushing forward is multiplication by $\alpha=[\P(F_3)]$.  Using the relation given in \S\ref{subsec:presentations}, we have
\begin{eqnarray*}
[\P(F_1)] &=& \alpha\cdot(x_1^2 + c_1(F_1)\, x_1 + c_2(F_3) - c_1(F_1)^2) \\
&=& \frac{1}{2}( x_1^3 - c_1(F_3)\, x_1^2 + c_2(F_3)\, x_1 - c_3(F_3) )( x_1^2 + c_1(F_1)\, x_1 + c_2(F_3) - c_1(F_1)^2 ).
\end{eqnarray*}
\end{proof}

Recall that the \emph{divided difference operators} for $G_2$ are defined as in \S\ref{introsec:formulas}, using the formulas  \eqref{eqn:divdiff-s} and \eqref{eqn:divdiff-t} for the operators $\partial_s$ and $\partial_t$ corresponding to simple reflections.  
These operators may be constructed geometrically, using a correspondence as described in \cite{flags}.  Let $\QQ(V)$ and $\GG(V)$ be the quadric bundle and bundle of $\gamma$-isotropic $2$-planes in $V$, respectively, and set $Z_s = \FFl_\gamma(V) \times_{\GG(V)} \FFl_\gamma(V)$ and $Z_t = \FFl_\gamma(V) \times_{\QQ(V)} \FFl_\gamma(V)$, with projections $p_i^s: Z_s \to \FFl_\gamma$ and $p_i^t: Z_t \to \FFl_\gamma$.  The proofs of the following two lemmas are the same as in classical types; see \cite[\S4.1]{thesis} for details.

\begin{lemma}
As maps $H^*\FFl_\gamma \to H^*\FFl_\gamma$,
\begin{eqnarray*}
\partial_s &=& (p_1^s)_*\circ(p_2^s)^* \text{ and} \\
\partial_t &=& (p_1^t)_*\circ(p_2^t)^*.
\end{eqnarray*}
\end{lemma}

\begin{lemma} \label{lemma:divdiff}
We have
\begin{eqnarray*}
\partial_s [\OOmega_w] &=& \left\{\begin{array}{cl} [\OOmega_{w\,s}] &\text{if } \ell(w\,s) < \ell(w); \\ 0 &\text{otherwise;} \end{array}\right.
\end{eqnarray*}
and
\begin{eqnarray*}
\partial_t [\OOmega_w] &=& \left\{\begin{array}{cl} [\OOmega_{w\,t}] &\text{if } \ell(w\,t) < \ell(w); \\ 0 &\text{otherwise.} \end{array}\right. 
\end{eqnarray*}
\end{lemma}

By making the substitutions $x_i \mapsto x_i + v$ and $y_i \mapsto y_i - v$, we obtain formulas for the more general case, where $\gamma$ has values in $M^{\otimes 3}$ for arbitrary $M$.

\begin{theorem} \label{thm:main}
Let $\gamma:\exterior^3 V \to M^{\otimes 3}$ be a nondegenerate form, with a $\gamma$-isotropic flag $F_\bullet \subset V$.  Let $v=c_1(M)$.  Let $\partial_s$ be defined as above, and let $\partial_t$ be given by
\begin{eqnarray}
\partial_t(f) &=& \frac{f(x_1,x_2) - f(x_1,x_1-x_2-v)}{-x_1 + 2x_2+v}. \label{eqn:divdiff-tv}
\end{eqnarray}
Then
\begin{eqnarray*}
[\OOmega_w] &=& \GP_w(x;y;v),
\end{eqnarray*}
where $\GP_w = \partial_{w_0\,w^{-1}} \GP_{w_0}$, and
\begin{eqnarray*}
\GP_{w_0}(x;y;v) &=& \frac{1}{2}( x_1^3 - 2\, x_1^2\, y_1 + x_1\, y_1^2 - x_1\, y_2^2 + x_1\, y_1\, y_2 - y_1^2\, y_2 + y_1\, y_2^2 \\
& & \quad + 5\, x_1^2\,v - 7\, x_1\, y_1\,v + x_1\, y_2\,v + 2\, y_1^2\,v + y_1\,y_2\,v - 2\, y_2^2\,v \\
& & \quad + 8\, x_1\,v^2 - 6\, y_1\,v^2 + 2\, y_2\,v^2 + 4\, v^3) \\
& & \times ( x_1^2 + x_1\, y_1 + y_1\, y_2 - y_2^2 + x_1\,v + y_2\, v) (x_2 - x_1 - y_2 + v).
\end{eqnarray*}
\end{theorem}

\section{Variations} \label{sec:variations}

Any formula for the class of a degeneracy locus depends on a choice of representative modulo the ideal defining the cohomology ring; here we discuss some alternative formulas.  In type $A$, the \emph{Schubert polynomials} of Lascoux and Sch\"utzenberger are generally accepted as the best polynomial representatives for Schubert classes and degeneracy loci: they have many remarkable geometric and combinatorial (and aesthetic) properties.  In other classical types, several choices have been proposed --- see, e.g., \cite{bh,lp,kt,fk,clgp} --- but Fomin and Kirillov \cite{fk} gave examples showing that no choice can satisfy all the properties possessed by the type $A$ polynomials, if one insists on having polynomials in the Chern roots of the tautological bundles.  
We suggest that an investigation of alternative $G_2$ formulas could shed some light on the problem for classical types, and vice versa.  On one hand, possibilities for $G_2$ formulas impose some limitations on what one might hope to find for general Lie types; on the other hand, one can use known formulas for type $B_3$, together with the embedding of $\Fl_\gamma$ in $\Fl_\beta$, to obtain new $G_2$ formulas.  We explore the latter point of view at the end of this section, using the polynomials of \cite{bh}.

\begin{proposition}[cf.\ \cite{g}] \label{prop:topclass}
Let
\begin{eqnarray*}
\tilde{\GP}_{w_0}(x;y) &=&
\frac {1}{54}\, ( 2\,x_1-x_2-y_1+2\,y_2
 )  ( 2\,x_1-x_2-y_1-y_2 )  ( x_
1-2\,x_2+y_1+y_2 ) \\
& & \times  ( 2\, x_1^3 - 3\,
x_1^2 x_2-3\, x_1 x_2^2 + 2\, x_2^3 - 2\, y_1^3 + 3\, y_1^2 y_2 + 3\, y_1 y_2^2 - 2\, y_2^3 ) .
\end{eqnarray*}
Then $[\OOmega_{w_0}] = \tilde{\GP}_{w_0}(x;y)$ in $H^*\FFl_\gamma(V)$.
\end{proposition}

\begin{proof}
Up to a change of variables, this is proved in \cite{g}.  (To recover Graham's notation, set 
\begin{eqnarray} \label{eqn:graham}
\renewcommand{\arraystretch}{1.5}
\begin{array}{lcl}
\xi_1 = \frac{1}{3}(2x_1 - x_2), & & \eta_1 = -\frac{1}{3}(2y_1 - y_2),\\
\xi_2 = \frac{1}{3}(-x_1+2x_2),  & & \eta_2 = -\frac{1}{3}(-y_1+2y_2),\\
\xi_3 = -\frac{1}{3}(x_1+x_2), &  & \eta_3 = \frac{1}{3}(y_1+y_2) ,
\end{array}
\renewcommand{\arraystretch}{1}
\end{eqnarray}
and replace $\xi,\eta$ with $x,y$.)
\end{proof}

\begin{remark} \label{rmk:graham}
In Graham's notation, $\tilde{\GP}_{w_0} = -\frac{27}{2}(\xi_1 - \eta_2)(\xi_1 - \eta_3)(\xi_2 - \eta_3)(\xi_1 \xi_2 \xi_3 + \eta_1 \eta_2 \eta_3)$.  This led him to suggest that $\frac{1}{2}(\xi_1 \xi_2 \xi_3 + \eta_1 \eta_2 \eta_3)$ might be an integral class.  In fact, only $27$ times this class is integral: Taking $[\Omega_w]^T = \tilde{\GP}_w(x;t) = \partial_{w_0 w^{-1}} \tilde{\GP}_{w_0}(x;t)$, we compute
\begin{eqnarray*}
\frac{1}{2}(\xi_1 \xi_2 \xi_3 + \eta_1 \eta_2 \eta_3) = -\frac{1}{27}\left( 3 [\Omega_{tst}]^T + 3 (t_1 + t_2) [\Omega_{st}]^T + (t_1 + t_2)(2t_1 - t_2) [\Omega_t]^T \right)
\end{eqnarray*}
in $H_T^*(\Fl_\gamma,\Q)$; here the $t$'s are related to the $\eta$'s as in \eqref{eqn:graham}.  (In fact, the two sides are equal as polynomials, not just as classes.)  Since the equivariant Schubert classes $[\Omega_w]^T$ form a basis for $H_T^*(\Fl_\gamma,\Z)$ over $H_T^*(pt,\Z) = \Z[t_1,t_2]$, the right-hand side cannot be integral.

It is interesting to note that the integral class $-\frac{27}{2}(\xi_1 \xi_2 \xi_3 + \eta_1 \eta_2 \eta_3)$ is \emph{positive} in the sense of \cite[Theorem 3.2]{g2}: the coefficients in its Schubert expansion are nonnegative combinations of monomials in the positive roots.  It is therefore natural to ask whether this is the equivariant class of a $T$-invariant subvariety of $\Fl_\gamma$.  In fact, it is the class of a $T$-equivariant embedding of $SL_3/B$.\footnote{This embedding projects to a $\P^2 \subset \GG$.  It is different from the embeddings of $SL_3/B$ corresponding to the inlcusion of Lie algebras $\liesl_3 \subset \lieg_2$, which project to $\P^2$'s in $\QQ$.}
\end{remark}

\begin{remark}
Graham's polynomial yields a simpler formula for the case where $\gamma$ takes values in the trivial bundle, but $\det V = M$ is not necessarily trivial.  (In this case, recall that $M^{\otimes 3}$ is trivial.)  Making the substitutions $x_i \mapsto x_i + v$ and $y_i \mapsto y_i - v$, with $3\,v=0$, we obtain
\begin{eqnarray*}
[\OOmega_{w_0}] &=& \frac {1}{54}\, (2\,x_1 - x_2 - y_1 + 2\,y_2)(2\,x_1 - x_2 - y_1 - y_2)(x_1 - 2\,x_2 + y_1 + y_2) \\
& & \times (2\,x_1^3 - 3\,x_1^2 x_2 - 3\,x_1 x_2^2 + 2\,x_2^3 - 2\,y_1^3 + 3\,y_1^2 y_2 + 3\,y_1 y_2^2 - 2\,y_2^3 + v^3).
\end{eqnarray*}
\end{remark}

\medskip

There is a more transparent choice of polynomial representative for $[\Omega_{w_0}] \in H^*\Fl_\gamma$ (i.e., the case where the base is a point):  The class of a point in the $5$-dimensional quadric $\QQ$ is $\frac{1}{2} x_1^5$.  Since $\Fl_\gamma$ is a $\P^1$ bundle over $\QQ$, and $x_2$ is the Chern class of the universal quotient bundle, the class of a point in $\Fl_\gamma$ is $[\Omega_{w_0}] = \frac{1}{2}x_1^5 x_2$.

Starting from $\bar{\GP}_{w_0} = \frac{1}{2}x_1^5 x_2$, one can compute polynomials $\bar{\GP}_w$ for Schubert classes ${[\Omega_w]}$ using divided difference operators.  
The resulting formulas are displayed in Table~\ref{table:gp}.

\begin{table}
\[
\begin{array}{|l|l|} \hline
 w          &  \bar{\GP}_w    \\ \hline\hline
 w_0  & \frac{1}{2}x_1^5 x_2  \\ \hline
ststs & \frac{1}{2}x_1^5   \\ \hline
tstst & \frac{1}{2}(x_1^3+x_2 x_1^2+x_2^2 x_1+x_2^3) x_1 x_2 \\ \hline
tsts  & \frac{1}{2}(4x_1^2 - 3x_1 x_2 + 3x_2^2) x_1^2  \\ \hline
stst  & \frac{1}{2}(x_1^4 + x_1^3 x_2 + x_1^2 x_2^2 + x_1 x_2^3 + x_2^4) \\ \hline
sts   & \frac{1}{2}(4x_1^2 - 3x_1 x_2 + 3x_2^2) x_1  \\ \hline
tst   & 2x_1^3 + \frac{1}{2} x_1^2 x_2 + \frac{1}{2} x_1 x_2^2 + 2x_2^3 \\ \hline
ts    & 3x_1^2 - 2x_1 x_2 + 2x_2^2  \\ \hline
st    & 2x_1^2 - x_1 x_2 + 2x_2^2 \\ \hline
s     & x_1  \\ \hline
t     & x_1 + x_2 \\ \hline
id    & 1 . \\ \hline
\end{array}
\]
\caption{Schubert polynomials for $Fl_\gamma$. \label{table:gp}}
\end{table}

The polynomials $\bar{\GP}_w$ computed from $\bar{\GP}_{w_0} = \frac{1}{2}x_1^5 x_2$ have negative coefficients.  
In fact, it is impossible to find a system of \emph{positive} polynomials using divided difference operators.  In this respect, the problem of ``$G_2$ Schubert polynomials'' is worse than the situation for types $B$ and $C$: they cannot even satisfy two of Fomin-Kirillov's conditions \cite{fk}.\footnote{To be precise, the conditions we consider are \cite[(3)]{fk} and a stronger version of \cite[(1)]{fk}.}  Specifically, we have the following:
\begin{proposition}
Let $\{P_w \,|\, w\in W\}$ be a set of homogeneous polynomials in the variables $x_1$ and $x_2$, with $\deg P_w = \ell(w)$.  Suppose
\begin{eqnarray*}
\partial_{s}P_w &=& \left\{ \begin{array}{cl} P_{ws} &\text{when } \ell(w\,s) < \ell(w) ; \\ 0 &\text{when } \ell(w\,s) > \ell(w) \end{array} \right.
\end{eqnarray*}
and
\begin{eqnarray*}
\partial_{t}P_w &=& \left\{ \begin{array}{cl} P_{wt} &\text{when } \ell(w\,t) < \ell(w) ; \\ 0 &\text{when } \ell(w\,t) > \ell(w). \end{array} \right.
\end{eqnarray*}
Then for some $w$, $P_w$ has both positive and negative coefficients.
\end{proposition}

\begin{proof}
One just calculates, starting from $P_{id} = 1$, and finds that the positivity requirement leaves no choice in the polynomials up to degree $4$:
\begin{eqnarray*}
 P_{w_0} &=& ? \\
P_{ststs} = ?&  \\
& & P_{tstst} = ? \\
P_{tsts} = ? & & \\
& & P_{stst} = \frac{1}{2} x_1^2 x_2^2 \\
P_{sts} = \frac{1}{2}x_1^3 & & \\
& & P_{tst} = \frac{1}{2}(x_1^2 x_2 + x_1 x_2^2) \\
P_{ts} = x_1^2 & & \\
& & P_{st} = \frac{1}{2}(x_1^2 + x_1 x_2 + x_2^2) \\
P_s = x_1 & & \\
& & P_t = x_1 + x_2 \\
 P_{id} &=& 1 .
\end{eqnarray*}
However, no degree $4$ polynomial $P = P_{tsts}$ satisfies all the hypotheses.  (Indeed, if $P = a x_1^4 + b x_1^3 x_2 + \cdots + e x_2^4$, then $\partial_t P = 0$ implies $d = -2e$ and $b+c+d+e=0$, hence $d=e=b=c=0$.  On the other hand, $\partial_s P = \frac{1}{2}(x_1^2 x_2 + x_1 x_2^2)$ requires $a=e$ and $b-d = \frac{1}{2}$, which is inconsistent with $b=c=d=e=0$.)
\end{proof}

In spite of this, one might look for polynomials which are positive in some other set of variables.  One natural choice is to use $x_2$ and $x_3 = x_1 - x_2$; in fact, the polynomials computed from $\bar{\GP}_{w_0} = \frac{1}{2}x_1^5 x_2 = \frac{1}{2}(x_2+x_3)^5 x_2$ are positive in these variables.

Finally, we consider an approach to $G_2$ Schubert polynomials using the polynomials $\mathfrak{B}^{\mathrm{BH}}_w$ defined for type $B$ flag varieties by Billey and Haiman \cite{bh}.  The Billey--Haiman polynomials involve a plethystic substitution, and the geometric meaning of the new variables is not altogether clear.  The change of variables has considerable advantages, though: the resulting polynomials are canonically defined, and they satisfy the conditions of \cite{fk} for the classical types.

Using the embedding of $\Fl_\gamma(V)$ in $\Fl_\beta(V)$, one can restrict the Billey--Haiman polynomials for type $B_3$ to obtain classes in $H^*Fl_\gamma$.  In fact, the polynomial $\mathfrak{B}^\mathrm{BH}_{\bar{2}\,\bar{3}\,1}$ restricts to the class $[\Omega_{w_0}]$, and applying the $G_2$ divided difference operators yields the appealing formulas displayed in Table~\ref{table:bh}; note that they are positive, have integral coefficients, and involve combinations of symmetric functions and type $A$ Schubert polynomials.\footnote{One of the other five degree $6$ type $B_3$ polynomials also restricts to $[\Omega_{w_0}]$, namely $\mathfrak{B}^\mathrm{BH}_{2\,1\,\bar{3}}$, but the resulting formulas are not positive.  The remaining four polynomials restrict to zero in $H^*Fl_\gamma$.}  However, they do not multiply exactly as Schubert classes, as one sees by comparing $\GP^\mathrm{BH}_s \cdot \GP^\mathrm{BH}_t$ with $\GP^\mathrm{BH}_{ts} + \GP^\mathrm{BH}_{st}$.  (The difference lies in the ideal defining $H^*Fl_\gamma$.)

To interpret these formulas, set $z_1 = x_2 - x_1$ and $z_2 = -x_2$.  For a strict partition $\lambda$, $P_\lambda = P_\lambda(\tilde{X})$ is the Schur $P$-function in a new set of variables $\tilde{X} = (\tilde{x}_1,\tilde{x}_2,\ldots)$, related to $Z=(z_1,z_2,\ldots)$ by the substitution $p_k(\tilde{X}) = -p_k(Z)/2$.  One can compute the table using an analogue of \cite[Corollary 4.5]{bh} to expand $\partial_s P_\lambda$.

\begin{table}[h]
\[
\begin{array}{|l|l|} \hline
w    &  \GP^{\mathrm{BH}}_w   \\ \hline\hline 
w_0  &  P_{42} + P_{32} z_1 \\ \hline
ststs & P_{32} \\ \hline
tstst & 2P_{41} + (P_4 + 2P_{31}) z_1 + P_3 z_1^2 \\ \hline
tsts & P_4 + 2P_{31} + P_3(z_1+z_2) \\ \hline
stst & 2P_{31} + (P_3 + 2P_{21})z_1 + P_2 z_1^2 \\ \hline
sts & P_3 + 2P_{21} + P_2(z_1 + z_2) \\ \hline
tst &  3P_3 + 4P_{21} + 4P_2 z_1 + 2P_2(z_1+z_2) + 4P_1 z_1^2 + 2P_1 z_1 z_2 + z_1^3 + z_1^2 z_2   \\ \hline
ts  &  4P_2 + 4P_1(z_1+z_2) + (z_1+z_2)^2  \\ \hline
st  &  3P_2 + 2P_1 z_1 + 2P_1(z_1+z_2) + z_1^2 + z_1 z_2\\ \hline
s   & 2P_1 + (z_1+z_2) \\ \hline
t   &  4P_1 + z_1 + 2(z_1+z_2)   \\ \hline
id   &  1 \\ \hline
\end{array}
\]
\caption{Type $G_2$ ``Billey--Haiman'' polynomials, computed from $\GP^{\mathrm{BH}}_{w_0} = \mathfrak{B}^{\mathrm{BH}}_{\bar{2}\,\bar{3}\,1}$. \label{table:bh}}
\end{table}


\appendix

\section{Lie theory} \label{ch:liethy}

In this appendix, we recall general facts about representation theory and homogeneous spaces for linear algebraic groups, and apply them show the above description of the $G_2$ flag variety agrees with the Lie-theoretic one.  Propositions \ref{prop:stabilizer} and \ref{prop:gamma-to-beta} are the basic representation-theoretic facts relating $G_2$ to compatible forms; their proofs are given in \cite[Appendix A]{thesis}.  Proposition \ref{prop:g2flag} identifies the $\gamma$-isotropic flag variety $\Fl_\gamma$ with the homogeneous space $G_2/B$.  
All the remaining facts are standard, and can be found in e.g. \cite{fh}, \cite{humphreys-repthy}, \cite{humphreys-gps}, \cite{d}.

\subsection{General facts} \label{sec:lie-general}

Let $G$ be a simple linear algebraic group, fix a maximal torus and Borel subgroup $T \subset B \subset G$, and let $W = N(T)/T$ be the Weyl group.  Let $R$, $R^+$, and $\Delta$ be the corresponding roots, positive roots, and simple roots, respectively.  For $\alpha\in\Delta$, let $s_\alpha\in W$ be the corresponding simple reflection, and also write $s_\alpha\in N(T)$ for a choice of lift; nothing in what follows will depend on the choice.  For a subset $S\subset \Delta$, let $P_S$ be the parabolic subgroup generated by $B$ and $\{s_\alpha \,|\, \alpha\in S\}$.  (Such parabolic subgroups are called \emph{standard}.)  
Write $\hat\imath = \Delta\setminus\{s_i\}$, so $P_{\hat\imath}$ is the maximal parabolic in which the $i$th simple root is omitted.  (For example, $SL_5/P_{\hat{2}} \isom Gr(2,5)$.)  Write $\lieg$, $\lieb$, $\liet$, $\liep$, for the corresponding Lie algebras.
 
The \define{length} of an element $w\in W$ is the least number $\ell=\ell(w)$ such that $w = s_1 \cdots s_\ell$ (with $s_j = s_{\alpha_j}$ for some $\alpha_j\in\Delta$); such a minimal expression for $w$ is called a \define{reduced expression}.  Let $w_0$ be the (unique) longest element of $W$.  The \define{Bruhat order} on $W$ is defined by setting $v\leq w$ if there are reduced expressions $v=s_{\beta_1}\cdots s_{\beta_{\ell(v)}}$ and $w=s_{\alpha_1}\cdots s_{\alpha_{\ell(w)}}$ such that the $\beta$'s are among the $\alpha$'s.

For each $w\in W$, there is a \define{Schubert cell} $X^o_w = BwB/B$ in $G/B$, of dimension $\ell(w)$.  The \define{Schubert varieties} $X_w$ are the closures of cells, and $X_v \subseteq X_w$ iff $v\leq w$.

The irreducible representations of $G$ are indexed by dominant weights; write $V_\lambda$ for the representation corresponding to the dominant weight $\lambda$.  In characteristic $0$, if $p_\lambda \in \P(V_\lambda)$ is the point corresponding to a highest weight vector, then $G\cdot p_\lambda$ is the unique closed orbit, and is identified with $G/P_{S(\lambda)}$, where $S(\lambda)$ is the set of simple roots orthogonal to $\lambda$ with respect to a $W$-invariant inner product.  In positive characteristic, $G/P_{S(\lambda)}$ can still be embedded in $\P(V)$ for some representation with highest weight $\lambda$, but $V$ need not be irreducible.  (See \cite[\S31]{humphreys-gps} for these facts about representations in arbitrary characteristic.)

\subsection{Representation theory of $G_2$} \label{sec:repthy}

The root system of type $G_2$ has simple roots $\alpha_1$ and $\alpha_2$ (with $\alpha_2$ the long root), and positive roots $\alpha_1,\, \alpha_2,\, \alpha_1 + \alpha_2,\, 2\alpha_1 + \alpha_2,\, 3\alpha_1 + \alpha_2,\, 3\alpha_1 + 2\alpha_2$.  
The lattice of abstract weights is the same as the root lattice (cf.\ \cite[\S A.9]{humphreys-gps}); it follows that up to isomorphism, there is only one simple group of type $G_2$ (over the algebraically closed field $k$).  From now on, let $G$ denote this group, and fix $T\subset B\subset G$ corresponding to the root data.  By Proposition \ref{prop:auts}, $G \isom \Aut(C)$, where $C$ is the unique octonion algebra over $k$.  Let $V = e^\perp \subseteq C$ be the imaginary subspace.

The dominant Weyl chamber for this choice of positive roots is the cone spanned by $\alpha_4$ and $\alpha_6$; denote these fundamental weights by $\omega_1$ and $\omega_2$, respectively.  One checks that $V$ has highest weight $\omega_1$, and is irreducible for $\Char(k)\neq 2$, so $V = V_{\omega_1}$ is the minimal irreducible representation, called the \define{standard} representation of $G$.\footnote{If $\Char(k) = 2$, the representation $V = e^\perp \subset C$ contains an invariant subspace spanned by $e$.  In this case, the irreducible representation $V_{\omega_1} = V/(k\cdot e)$ is $6$-dimensional \cite[\S2.3]{sv}.}  The adjoint representation $\lieg$ has highest weight $\omega_2$.  (This is irreducible if $\Char(k)=0$, but not if $\Char(k)=3$.)  Over any field, one has $\lieg \subseteq \exterior^2 V$.

Let $\gamma$ be the alternating trilinear form on $V\subset C$ induced by the multiplication, let $\{f_1,\ldots,f_7\}$ be the standard $\gamma$-isotropic basis \eqref{f-basis}.  From the description of $G$ as the automorphisms of $C$, it is clear that $G$ preserves $\gamma$.  In fact, the converse is almost true:

\begin{proposition} \label{prop:stabilizer}
Choose a basis $\{f_1,\ldots,f_7\}$ for $V$, and let $\gamma \in \exterior^3 V^*$ be given by
\begin{eqnarray*}
\gamma = f_{147}^* + f_{246}^* + f_{345}^* - f_{156}^* - f_{237}^*,
\end{eqnarray*}
as in \eqref{eqn:f-gamma}.  Let $G(\gamma) \subset GL(V)$ be the stabilizer of $\gamma$ under the natural action, and let $SG(\gamma) = G(\gamma) \cap SL(V)$.  Then $SG(\gamma)$ is simple of type $G_2$, and $G(\gamma) = \mu_3\times SG(\gamma)$.  Moreover, the orbit $GL(V)\cdot\gamma$ is open in $\exterior^3 V^*$.
\end{proposition}

\noindent
For $k = \C$, this is well known; see \cite[\S2]{bryant} or \cite[\S22]{fh}.  For arbitrary fields, see \cite[Propositions 6.1.4 and A.2.2]{thesis}, and compare \cite[(3.4)]{asch}.  

The proof of this proposition also shows the following:
\begin{corollary}
Let $V$, $\gamma$, and $SG(\gamma)$ be as in Proposition \ref{prop:stabilizer}.  Then $SG(\gamma)$ acts irreducibly on $V$. \qedhere\qed
\end{corollary}

Note that $w_0\in W$ acts on the weight lattice by multiplication by $-1$.  This implies that every irreducible representation of $G$ is isomorphic to its dual.  Using Schur's lemma, there is a unique (up to scalar) $G$-invariant bilinear form on each irreducible representation \cite[\S31.6]{humphreys-gps}.  In particular, we have the following:

\begin{proposition} \label{prop:gamma-to-beta}
Let $V$ be a $7$-dimensional vector space, with nondegenerate trilinear form $\gamma:\exterior^3 V \to k$.  Then $\gamma$ determines a compatible form $\beta$ uniquely up to scaling by a cube root of unity. \qedhere\qed
\end{proposition}

\begin{remark} \label{rmk:history}
In characteristic $0$, the description of $G_2$ (or $\lieg_2$) as the stabilizer of a generic alternating trilinear form is due to Engel, who also found an invariant symmetric bilinear form.  For a history of some of the early constructions of $G_2$, see \cite{agricola}.
\end{remark}

\subsection{The Weyl group} \label{sec:weyl}

The Weyl group of type $G_2$ is the dihedral group with $12$ elements.  Let $\alpha_1$ and $\alpha_2$ be the simple roots, and let $s=s_{\alpha_1}$ and $t=s_{\alpha_2}$ be the corresponding simple reflections generating $W=W(G_2)$.  Thus $W$ has a presentation $\<s,t\,|\, s^2=t^2=(st)^6=1 \>$.  With the exception of $w_0$, each element of $W(G_2)$ has a unique reduced expression.  The Hasse diagram for Bruhat order is as follows:

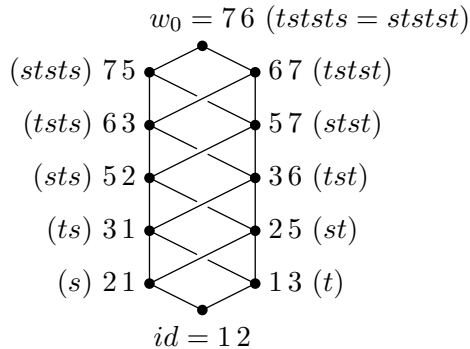
\pspicture(-170,-40)(100,100)
\pscircle*(0,80){2}
\rput[l](-20,90){$w_0=7\,6$ ($tststs=ststst$)}

\pscircle*(-20,70){2}
\pscircle*(20,70){2}
\psline(0,80)(-20,70)
\psline(0,80)(20,70)
\rput[r](-25,70){($ststs$) $7\,5$}
\rput[l](25,70){$6\,7$ ($tstst$)}

\pscircle*(-20,50){2}
\psline(-20,70)(-20,50)
\psline(20,70)(-20,50)
\pscircle*(20,50){2}
\psline(20,70)(20,50)
\psline(-20,70)(-2,61)
\psline(2,59)(20,50)
\rput[r](-25,50){($tsts$) $6\,3$}
\rput[l](25,50){$5\,7$ ($stst$)}

\pscircle*(-20,30){2}
\psline(-20,50)(-20,30)
\psline(20,50)(-20,30)
\pscircle*(20,30){2}
\psline(20,50)(20,30)
\psline(-20,50)(-2,41)
\psline(2,39)(20,30)
\rput[r](-25,30){($sts$) $5\,2$}
\rput[l](25,30){$3\,6$ ($tst$)}

\pscircle*(-20,10){2}
\psline(-20,30)(-20,10)
\psline(20,30)(-20,10)
\pscircle*(20,10){2}
\psline(20,30)(20,10)
\psline(-20,30)(-2,21)
\psline(2,19)(20,10)
\rput[r](-25,10){($ts$) $3\,1$}
\rput[l](25,10){$2\,5$ ($st$)}

\pscircle*(-20,-10){2}
\psline(-20,10)(-20,-10)
\psline(20,10)(-20,-10)
\pscircle*(20,-10){2}
\psline(20,10)(20,-10)
\psline(-20,10)(-2,1)
\psline(2,-1)(20,-10)
\rput[r](-25,-10){($s$) $2\,1$}
\rput[l](25,-10){$1\,3$ ($t$)}

\pscircle*(0,-20){2}
\psline(-20,-10)(0,-20)
\psline(20,-10)(0,-20)
\rput(0,-30){$id=1\,2$}

\endpspicture

The indexing $w = w(1)\,w(2)$, for $1\leq w(1),w(2)\leq 7$, arises as follows.  There is an embedding $W(G_2)\hookrightarrow W(A_6) = S_7$, given by $s\mapsto \tau_{12}\tau_{35}\tau_{67}$ and $t\mapsto \tau_{23}\tau_{56}$, where $\tau_{ij}$ is the permutation transposing $i$ and $j$.  (This also factors through $W(B_3)$.)  Thus each $w$ is identified with a permutation $w(1)\,w(2)\,\cdots w(7)$, and in fact, the full permutation is determined by $w(1)\,w(2)$.

This inclusion of Weyl groups corresponds to the inclusion $G_2\hookrightarrow SL_7$ determined by the basis $\{f_1,\ldots,f_7\}$ for $V=V_{\omega_1}$ and the trilinear form $\gamma$ of \eqref{eqn:f-gamma}, together with the inclusion of tori $(z_1,z_2)\mapsto (z_1, z_2, z_1 z_2^{-1}, 1, z_1^{-1} z_2, z_2^{-1}, z_1^{-1})$.  Thus a natural way to extend $w\in W$ to a full permutation is as follows.  Given $w(1)\,w(2)$, let $w(3)$ be the number such that $E_{f_{w(1)}} = \< f_{w(1)}, f_{w(2)}, f_{w(3)} \>$ as in \S \ref{subsec:fixed}.  Then define $w(4),\ldots,w(7)$ by requiring $w(i) + w(8-i) = 8$.  For example, $6\,3$ extends to $6\,3\,7\,4\,1\,5\,2$.  Note that $(w\cdot w_0)(i) = 8-w(i)$.

All this can be summarized in the following diagram:

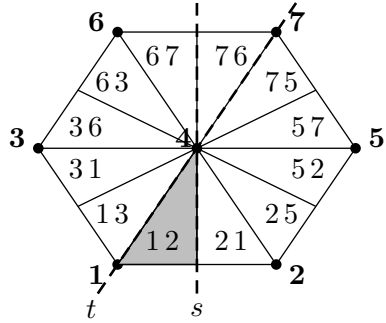
\pspicture(-160,-60)(100,65)

\pscircle*(-30,-44){2}
\rput(-38,-48){$\mathbf{1}$}
\pscircle*(30,-44){2}
\rput(38,-48){$\mathbf{2}$}
\pscircle*(-60,0){2}
\rput(-68,4){$\mathbf{3}$}
\pscircle*(0,0){2}
\rput(-5,4){$\mathbf{4}$}
\pscircle*(60,0){2}
\rput(68,4){$\mathbf{5}$}
\pscircle*(-30,44){2}
\rput(-38,48){$\mathbf{6}$}
\pscircle*(30,44){2}
\rput(38,48){$\mathbf{7}$}

\pspolygon[fillstyle=solid,fillcolor=lightgray](-30,-44)(0,0)(0,-44)

\psline(-30,-44)(30,-44)(60,0)(30,44)(-30,44)(-60,0)(-30,-44)

\psline(-30,-44)(30,44)
\psline(30,-44)(-30,44)
\psline(-60,0)(60,0)

\psline(0,-44)(0,44)
\psline(45,-22)(-45,22)
\psline(-45,-22)(45,22)

\rput(-13,-35){$1\,2$}
\rput(13,-35){$2\,1$}
\rput(-32,-25){$1\,3$}
\rput(32,-25){$2\,5$}
\rput(-42,-8){$3\,1$}
\rput(42,-8){$5\,2$}
\rput(-42,8){$3\,6$}
\rput(42,8){$5\,7$}
\rput(-32,25){$6\,3$}
\rput(32,25){$7\,5$}
\rput(-13,35){$6\,7$}
\rput(13,35){$7\,6$}

\psline[linewidth=1,linestyle=dashed](0,-55)(0,55)
\psline[linewidth=1,linestyle=dashed](-37.5,-55)(37.5,55)

\rput(0,-61){$s$}
\rput(-40,-61){$t$}

\endpspicture

\subsection{Homogeneous spaces} \label{sec:homogeneous}

We can now identify the homogeneous spaces for $G_2$.  We take $G = \Aut(C)$ for an octonion algebra $C$, as above, and let $\beta$ and $\gamma$ be the corresponding compatible forms on the imaginary subspace $V \subset C$.  From the root data, one sees $\dim G = 14$, $\dim B = 8$, $\dim P_{\hat{1}}=\dim P_{\hat{2}} = 9$, and $\dim T = 2$.  Thus $\dim G/B = 6$ and $\dim G/P_{\hat{1}} = \dim G/P_{\hat{2}} = 5$.

\begin{proposition} \label{prop:g2flag}
Let $\Fl_\gamma$, $\QQ$, and $\GG$ be as in \S \ref{sec:topology}.  Then $\QQ \isom G/P_{\hat{1}}$, $\GG \isom G/P_{\hat{2}}$, and $\Fl_\gamma \isom G/B$.
\end{proposition}

\begin{proof}
The homogeneous spaces $G/P_{\hat{1}}$ and $G/P_{\hat{2}}$ are the closed orbits in $\P(V)$ and $\P(\lieg)$, respectively.  Since $G$ preserves $\beta$, $G/P_{\hat{1}}$ must be contained in the quadric hypersurface $\QQ\subset\P(V)$, but $\dim G/P_{\hat{1}} = 5$, so it is all of $\QQ$.

To see $G/P_{\hat{2}} = \GG$, note that $G/P_{\hat{2}} \subset \P(\lieg) \subset \P(\exterior^2 V)$, so $G/P_{\hat{2}} \subset Gr(2,7)$.  Since $G$ preserves $\gamma$, we must have $G/P_{\hat{2}} \subseteq \GG$; thus it will suffice to show $\GG$ is irreducible and $5$-dimensional.  For this, consider
\begin{eqnarray*}
\Fl_\gamma = \{(p,\ell)\,|\, p\in\ell\} \subset \QQ \times \GG,
\end{eqnarray*}
and notice that the first projection identifies $\Fl_\gamma$ with the $\P^1$-bundle $\P(S_3/S_1) \to \QQ$, so $\Fl_\gamma$ is smooth and irreducible of dimension $6$.  On the other hand, the second projection is obviously a $\P^1$-bundle.

Finally, since $\Fl_\gamma$ is a $6$-dimensional $G$-invariant subvariety of $G/P_{\hat{1}} \times G/P_{\hat{2}}$, it follows that $\Fl_\gamma = G/B$.
\end{proof}

\begin{remark}
A similar description of $G/P_{\hat{2}}$, among others, can be found in \cite{lm2}.
\end{remark}

\begin{proposition} \label{prop:sch-include}
Let $i:G \hookrightarrow G'$ be an inclusion of semisimple algebraic groups, and let $B \subset G$ and $B'\subset G'$ be Borel subgroups with $i(B)\subset B'$.  Also denote by $i$ the induced inclusions of flag varieties $G/B\hookrightarrow G'/B'$ and Weyl groups $W \hookrightarrow W'$.  Then for each $w\in W$, the Schubert cells are related by $BwB/B = (B'i(w)B'/B') \cap (G/B)$.

More generally, let $P\subset G$ and $P'\subset G'$ be parabolic subgroups such that $P = P' \cap G$.  Then the same conclusion holds for $G/P \hookrightarrow G'/P'$, that is, $BwP/P = (B'i(w)P'/P') \cap (G/P)$ for all $w\in W$.
\end{proposition}

\subsection{The Borel map and divided differences} \label{sec:borel}

Let $M\subset \liet^*$ be the weight lattice.  For general $G/B$, there is a \define{Borel map}
\begin{eqnarray*}
c: \Sym^*M \to H^*(G/B)
\end{eqnarray*}
induced by the Chern class map $c_1:M \to H^2(G/B)$, where $M\subset\liet^*$ is the weight lattice.  More precisely, this map is defined as follows.  Identify $M$ with the character group of $B$, and associate to $\chi\in M$ the line bundle $L_{\chi} = G \times^B \C$.  Then $c_1(\chi)$ is defined to be $c_1(L_{\chi})$.  (See \cite{bgg,d}.)  In fact, $c_1$ is an isomorphism, and this induces an action of $W$ in the evident way: for $w\in W$ and $x = c_1(\chi) \in H^2(G/B)$, define $w\cdot x = c_1(w\cdot\chi)$.

The Borel map becomes surjective after extending scalars to $\Q$, and defines an isomorphism
\begin{eqnarray*}
H^*(G/B,\Q) \isom \Sym^* M_{\Q} / I,
\end{eqnarray*}
where $I = (\Sym^* M_{\Q})^W_+$ is the ideal of positive-degree Weyl group invariants.

For a simple root $\alpha$, define the \define{divided difference operator} $\partial_{\alpha}$ on $H^*(G/B)$ by
\begin{eqnarray} \label{eqn:divdiff-gen}
\partial_{\alpha}(f) = \frac{f - s_{\alpha}\cdot f}{\alpha}.
\end{eqnarray}
These act on Schubert classes as follows \cite{d}:
\begin{eqnarray} \label{eqn:div-diff-on-sch}
\partial_{\alpha}{[\Omega_w]} &=& \left\{ \begin{array}{cl} {[\Omega_{w\,s_{\alpha}}]} &\text{when } \ell(w\,s_{\alpha}) < \ell(w) ; \\ 0 &\text{when } \ell(w\,s_{\alpha}) > \ell(w). \end{array} \right.
\end{eqnarray}
In particular, ${[\Omega_{s_{\alpha}}]}$ can be identified with the weight at the intersection of the hyperplanes orthogonal to $\alpha$ and the (affine) hyperplane bisecting $\alpha$.

In the case of $G_2$ flags, we know ${[\Omega_s]} = x_1$ and ${[\Omega_t]} = x_1 + x_2$.  Looking at the root diagram, then, we see $x_1 = \alpha_4$ and $x_2 = \alpha_3$.  Therefore
\begin{eqnarray*}
\alpha_1 = x_1 - x_2, \quad \alpha_2 = -x_1 + 2x_2,
\end{eqnarray*}
and
\begin{eqnarray*}
s\cdot x_1 = x_2 , \quad
s\cdot x_2 = x_1 , \quad
t\cdot x_1 = x_1 , \quad
t\cdot x_2 = x_1 - x_2 .
\end{eqnarray*}
With these substitutions, the operators of \eqref{eqn:divdiff-gen} agree with those defined in \S\ref{sec:divdiff} (\eqref{eqn:divdiff-s} and \eqref{eqn:divdiff-t}).

\section{Integral Chow rings of quadric bundles} \label{ch:chow}

In this appendix, we consider schemes over an arbitrary field $k$, and use the language of Chow rings rather than cohomology.  We prove the following fact about odd-rank quadric bundles:

\begin{theorem} \label{thm:chow}
Let $V$ be a vector bundle of rank $2n+1$ on a scheme $X$, and suppose $V$ is equipped with a nondegenerate quadratic form.  Assume there is a maximal (rank $n$) isotropic subbundle $F \subset V$.  Let $\QQ \xrightarrow{p} X$ be the quadric bundle of isotropic lines in $V$, let $h \in A^*\QQ$ be the hyperplane class (restricted from $H=c_1(\O(1)) \in A^*\P(V)$), and let ${f} = [\P(F)] \in A^*\QQ$.  Then
\begin{eqnarray*}
A^*\QQ = A^*X [h,{f}]/I,
\end{eqnarray*}
where the ideal $I$ is generated by the two relations
\begin{eqnarray}
2{f} &=& h^n - c_1(F)\,h^{n-1} + \cdots + (-1)^n c_n(F), \label{rel1} \\
{f}^2 &=& (c_n(V/F) + c_{n-2}(V/F)\,h^2 + \cdots )\,{f}. \label{rel2}
\end{eqnarray}
(Here $h$ and ${f}$ have degrees $1$ and $n$, respectively.)
\end{theorem}

\noindent
A similar presentation for even-rank quadrics was first given by Edidin and Graham \cite[Theorem 7]{eg}; in fact, the second of the two relations is the same as theirs.  Our purpose here is to correct a small error in the statement of the second half of their theorem (which concerned odd-rank quadrics).

Before giving the proof, we recall two basic formulas for Chern classes.  Let $L$ be a line bundle.  For a vector bundle $E$ of rank $n$, we have (cf.\ \cite[Ex.\ 3.2.2]{it})
\begin{eqnarray} \label{eqn:tensor-line}
c_n(E\otimes L) = \sum_{i=0}^n c_i(E)\,c_1(L)^{n-i}.
\end{eqnarray}
Also, if
\begin{eqnarray*}
0 \to L \to E \to E' \to 0
\end{eqnarray*}
is an exact sequence of vector bundles, then inverting the Whitney formula gives
\begin{eqnarray} \label{eqn:inv-whitney}
c_k(E') = c_k(E) - c_{k-1}(E)\,c_1(L) + \cdots + (-1)^k c_1(L)^k.
\end{eqnarray}

\begin{proof}
The classes $h,h^2,\ldots,h^{n-1},{f},{f}\,h,\ldots,{f}\,h^{n-1}$ form a basis of $A^*\QQ$ as an $A^*X$-module, since they form a basis when restricted to a fiber.  It is easy to see that these elements also form a basis of the ring $A^*X[h,{f}]/I$.  Therefore it suffices to establish that the relations generating $I$ hold in $A^*\QQ$.

Let $i:\QQ \hookrightarrow \P(V)$ be the inclusion of the quadric in the projective bundle.  By \cite[Ex.\ 3.2.17]{it}, we have
\begin{eqnarray*}
i_*{f} = [\P(F)] = \sum_{i=0}^{n+1} c_i\,H^{n+1-i}
\end{eqnarray*}
in $A^*\P(V)$, where $c_i=c_i(V/F)$.  (Following the common abuse of notation, we have written $c_i$ for $p^*c_i$.)  On the other hand, $\QQ \subset \P(V)$ is cut out by a section of $\O_{\P(V)}(2)$, so $[\QQ] = 2\,H$ in $A^*\P(V)$.  Therefore $i^* i_*{f} = 2\,h\,{f}$, and we have
\begin{eqnarray} \label{eqn:eg-rel-1}
2\,h\,{f} &=& h^{n+1} + c_1\, h^n + \cdots + c_{n+1}.
\end{eqnarray}
(Up to this point, we are repeating the argument of \cite{eg}.)

To prove the first relation, expand $h^n$ in the given basis:
\begin{eqnarray} \label{eqn:mystery-rel}
h^n = a_0\,{f} + a_1\,h^{n-1} + \cdots + a_n,
\end{eqnarray}
with $a_k \in A^k X$.  Our goal is to show $a_0 = 2$, and $a_k = (-1)^{k+1}c_k(F)$ for $k>0$.

That $a_0 = 2$ can be seen by restricting to a fiber: the Chow ring of an odd-dimensional quadric in projective space is given by $\Z[h,{f}]/(h^n-2{f},{f}^2)$.

Multiplying \eqref{eqn:mystery-rel} by $h$ and expanding in the basis, we have
\begin{eqnarray*}
h^{n+1} = 2\,h\,{f} + 2\,a_1\,{f} + (a_2+a_1^2)\,h^{n-1} + \cdots + (a_n+a_1\,a_{n-1})\,h + a_1\,a_n.
\end{eqnarray*}
On the other hand, if we rearrange and expand \eqref{eqn:eg-rel-1}, we obtain
\begin{eqnarray*}
h^{n+1} = 2\,h\,{f} - 2\,c_1\,{f} - (c_2 + c_1\,a_1) h^{n-1} - \cdots - (c_n + c_1\,a_{n-1})\, h - (c_{n+1} + c_1\,a_n).
\end{eqnarray*}
Comparing coefficients, we have
\begin{eqnarray*}
2\,a_1 &=& -2\,c_1 ; \\
a_k &=& -c_k - a_{k-1}(a_1 + c_1) \quad (2\leq k\leq n) ; \\
a_1\,a_n &=& - c_{n+1} - c_1\,a_1.
\end{eqnarray*}
From the first of these equations, we see
\begin{eqnarray*}
a_1 + c_1 = \tau,
\end{eqnarray*}
for some $\tau\in A^1 X$ such that $2\,\tau=0$.  (Note that $\tau=0$ only if $c_{n+1}(V/F)=0$, which need not be true in general.)  The remaining equations give
\begin{eqnarray} \label{eqn:almost}
a_k &=& -c_k + c_{k-1}\,\tau - c_{k-2}\,\tau^2 + \cdots - (-1)^k \tau^k \quad (1\leq k\leq n),
\end{eqnarray}
and $-c_{n+1} = a_n\,\tau$.  (Of course, the signs on powers of $\tau$ make no difference, but we will include them as a visual aid.)

We claim $\tau = c_1(F^\perp/F)$.  This can be proved in the universal case.  Specifying the maximal isotropic subbundle $F\subset V$ reduces the structure group from $O_{2n+1}$ to a parabolic subgroup whose Levi factor is $GL_n \times \Z/2\Z$, so the universal base is (an affine bundle over) $BG=BGL_n \times B\Z/2\Z$.  Every such maximal isotropic subbundle $F\subset V$ on $X$ is pulled back from a universal subbundle $\tilde{F}\subset\tilde{V}$ on the classifying space $BG$.  More precisely, one should use Totaro's algebraic model for $BG$; to ensure $V$ is pulled back from the corresponding $\tilde{V}$ on the algebraic model, one may have to replace $X$ by an affine bundle or Chow envelope, as in \cite[p.\ 486]{g}.

Now $A^*(BGL_n \times B\Z/2\Z) \isom \Z[c_1,\ldots,c_n,t]/(2t)$, so there is only one nonzero $2$-torsion class of degree $1$, namely $t$.  Since $t=c_1(\tilde{F}^\perp/\tilde{F})$, it pulls back to $c_1(F^\perp/F)$, so the claim is proved.  (The meaning of the Chow ring of $BG$  is explained in \cite{t}, as is its computation.  To see that $t=c_1(\tilde{F}^\perp/\tilde{F})$, note that the inclusion $GL_n\times\Z/2\Z\subset O_{2n+1}\subset GL_{2n+1}$ also factors as $GL_n\times \Z/2\Z \subset GL_n\times \mathbb{G}_m \subset GL_{2n+1}$, corresponding to the splitting $\tilde{V}\isom\tilde{F}\oplus(\tilde{F}^\perp/\tilde{F})\oplus\tilde{F}^*$.)

Using the exact sequence $0\to F^\perp/F \to V/F \to V/F^\perp \to 0$ and Formula \eqref{eqn:inv-whitney}, Equation \eqref{eqn:almost} implies
\begin{eqnarray*}
a_k = - c_k(V/F^\perp).
\end{eqnarray*}
Since $V/F^\perp \isom F^*$, we obtain $a_k = (-1)^{k+1} c_k(F)$, as desired.

\smallskip

The second relation is proved by the argument given in \cite{eg}.  Let $j:\P(F) \hookrightarrow \QQ$ be the inclusion, and let $N_{\P(F)/\QQ}$ be the normal bundle.  By the self-intersection formula, $j_*c_n(N_{\P(F)/\QQ}) = {f}^2$.  On the other hand, using $N_{\QQ/\P(V)} = \O(2)$ and $N_{\P(F)/\P(V)} = V/F \otimes \O(1)$, and tensoring with $\O(-1)$, we have
\begin{eqnarray*}
0 \to \O(1) \to V/F  \to N_{\P(F)/\QQ}\otimes\O(-1) \to 0
\end{eqnarray*}
on $\P(F)$; thus $N_{\P(F)/\QQ} = ((V/F)/\O(1)) \otimes \O(1)$.  By Formulas \eqref{eqn:tensor-line} and \eqref{eqn:inv-whitney}, we have
\begin{eqnarray*}
c_n(N_{\P(F)/\QQ}) = c_n(V/F) + c_{n-2}(V/F)\,h^2 + \cdots.
\end{eqnarray*}
The relation \eqref{rel2} follows after applying $j_*$.
\end{proof}

\end{document}